\documentclass[leqno]{amsart}
\usepackage{amsmath,amsfonts,amsthm,amssymb,indentfirst,epic,url,graphics,needspace}

\newtheorem{theorem}{Theorem}
\newtheorem{lemma}[theorem]{Lemma}
\newtheorem{corollary}[theorem]{Corollary}
\newtheorem{proposition}[theorem]{Proposition}
\newtheorem{definition}[theorem]{Definition}
\newtheorem{question}[theorem]{Question}

\newcommand{\Z}{\mathbb{Z}}
\renewcommand{\r}{\mathrm}

\newcommand{\F}{\mathcal{F}}

\newcommand{\lang}{\begin{picture}(5,7)
\put(1.2,2.5){\rotatebox{45}{\line(1,0){6.0}}}
\put(1.2,2.5){\rotatebox{315}{\line(1,0){6.0}}}
\end{picture}\kern.16em}
\newcommand{\rang}{\kern.1em\begin{picture}(5,7)
\put(.1,2.5){\rotatebox{135}{\line(1,0){6.0}}}
\put(.1,2.5){\rotatebox{225}{\line(1,0){6.0}}}
\end{picture}}

\begin{document}

\begin{center}
\texttt{Comments, corrections,
and related references welcomed, as always!}\\[.5em]
{\TeX}ed \today
\vspace{2em}
\end{center}

\title[Strong inner inverses]%
{Strong inner inverses in endomorphism rings\\ of vector spaces}%
\thanks{\url{http://arxiv.org/abs/1611.00972}\,.\\
\hspace*{1.6em}After publication of this note,
updates, errata, related references etc., if found, will be recorded at
\url{http://math.berkeley.edu/~gbergman/papers/}\,.
}

\subjclass[2010]{Primary: 16E50, 16S50, 16U99, 20M18.
Secondary: 16S15, 16S36.}
\keywords{endomorphism ring of a vector space;
inner inverse to a ring element; inverse monoid.}

\author{George M. Bergman}
\address{University of California\\
Berkeley, CA 94720-3840, USA}
\email{gbergman@math.berkeley.edu}

\begin{abstract}
For $V$ a vector space over a field, or more generally, over a
division ring, it is well-known that every $x\in\r{End}(V)$
has an {\em inner inverse}; that is, that there exists
$y\in\r{End}(V)$ satisfying $xyx=x.$
We show here that a large class of
such $x$ have inner inverses $y$ that satisfy
with $x$ an infinite family
of additional monoid relations, making the monoid generated by
$x$ and $y$ what is known as an {\em inverse monoid}
(definition recalled).
We obtain consequences of these relations, and related results.

P.\,Nielsen and J.\,\v{S}ter~\cite{PN+JvS} show that a much
larger class of elements $x$ of rings $R,$ including all elements
of von Neumann regular rings, have inner inverses satisfying
arbitrarily large {\em finite} subsets of the abovementioned set of
relations.
But we show by example that the
endomorphism ring of any infinite-dimensional vector space
contains elements having no inner inverse that
simultaneously satisfies all those relations.

A tangential result gives a condition on an endomap $x$ of a
set $S$ that is necessary and sufficient for $x$ to have a strong
inner inverse in the monoid of all endomaps of $S.$
\end{abstract}
\maketitle

\section{Background}\label{S.bkgd}

A basic property of the endomorphism ring $R=\r{End}(V)$ of a
vector space $V$ over a division
ring is that for every $x\in R,$ there exists a $y\in R$ such that
\begin{equation}\begin{minipage}[c]{35pc}\label{d.xyx=x}
$xyx\ =\ x.$
\end{minipage}\end{equation}
Such a $y$ is called an ``inner inverse'' to $x.$
(Note that for any endomap $x$ of any set $S,$
an inner inverse to $x,$ i.e., an endomap $y$ of $S$
which satisfies~\eqref{d.xyx=x},
is simply a map that carries every element of the image of $x$
to a preimage of itself under $x.$
For $V$ a vector space and $x\in\r{End}(V),$
an inner inverse to $x$ in $\r{End}(V)$ can be constructed
by mapping each member of a vector-space basis $B_0$ of $x(V)$ to
a preimage of itself, and acting arbitrarily on the remaining
elements of a basis $B\supseteq B_0$ of $V.)$
If $R$ is a ring and $y\in R$ is an inner inverse of
$x\in R,$ we see from~\eqref{d.xyx=x} that
$xy$ and $yx$ are idempotent elements of $R.$

A ring such as $\r{End}(V)$ in which every element has an
inner inverse is called {\em von Neumann regular},
often shortened to {\em regular}.

In general, if $y$ is an inner inverse to $x,$ this does not
make $x$ an inner inverse to $y.$
(For instance, every element of a ring is an inner inverse to
the element $0.)$
However, if an element $x$ has an inner inverse
$y_0,$ it also has an inner inverse $y$ to which it is, itself,
an inner inverse, namely $y=y_0 xy_0.$

I had naively thought that $yxy=y$ was the strongest additional
relation one could hope to ask of an inner inverse $y$
to an element $x$ in a general regular ring.
Hence I was surprised to see, in an unpublished note
by Kevin O'Meara, a construction, from any $x$ in such a ring, of
an inner inverse $y$ which not only satisfies this additional relation,
but also the relation saying that the idempotents $xy$ and
$yx$ commute.
Note that this commutativity relation, $xyyx=yxxy,$ implies that
$xxyyxx=x(yxxy)x=(xyx)(xyx)=xx;$ in other words, that
$y^2$ is an inner inverse to $x^2;$ and by the symmetry of the
relations satisfied, $x^2$ is also an inner inverse to $y^2.$

This suggested that one try to obtain an inner inverse
$y$ to $x$ that satisfied further relations making $y^n$ an inner
inverse to $x^n$ for higher $n.$
And indeed, looking at a particularly nice class of regular
rings, namely direct products of matrix rings (in general,
of unbounded sizes) over fields, I was able to find an
infinite family of relations that an inner inverse to an element
of such a ring can be made to satisfy, which includes
the relations $x^n y^n x^n = x^n$ and $y^n x^n y^n = y^n.$

A second surprise came when I mentioned this to Pace Nielsen.
It turned out that in a paper \cite{PN+JvS} that
he and Janez \v{S}ter had submitted for publication, they showed
that for any ring $R,$ any element $x\in R,$ and any positive
integer $n$ such that $x,\ x^2,\dots,x^n$ all have
inner inverses, one can find an inner inverse $y$ to $x$
satisfying many of the same relations that I had found; in particular,
such that $y^j$ is an inner inverse to $x^j$ for all $j\leq n.$
After this discussion, they were able to extend the
relations they obtained so that the union, over all $n,$
of their families of relations coincided with my family.
(Their result appears, so revised, as \cite[Theorem~4.8]{PN+JvS}.
Moreover, the form in which they had originally
formulated their relations led to improvements in the present note.
So \cite{PN+JvS} and this note each contain material
inspired by the other.)

The family of relations in $x$ and $y$ referred to above are
monoid relations, and
a third surprise was to discover, after drafting most of
these pages, that the monoid $\F$ that they define is
well known to semigroup theorists, as the free {\em inverse monoid}
on one generator $x,$ and that there is considerable
literature on inverse monoids and their monoid algebras,
e.g., \cite{MP}, \cite{MVL}, parts of \cite{C+P1} and \cite{C+P2},
\cite[\S 4]{PA+KG}, \cite{C+M}, and \cite{WDM_rings}.

In~\S\ref{S.monoid} below, I develop that monoid
essentially ``as I found it'' (though I have borrowed the
notation $\F$ from~\cite[\S 4]{PA+KG}).
In~\S\ref{S.inv_md}, the concept of inverse monoid
is sketched, and in \S\ref{S.F} we verify the characterization
of $\F$ in terms of that concept.
In subsequent sections, though $\F$ is referred to in those terms,
the results are mostly independent of the literature on the subject.

\section{The monoid $\F$}\label{S.monoid}

Let us motivate the monoid we will be studying by
considering a natural family of pairs of mutually inner inverse
vector-space maps.
We write maps to the left of their arguments,
and compose them accordingly.

Suppose $V$ is an $\!n\!$-dimensional vector space with basis
$\{b_1,\dots,b_n\},$ let $x:V\to V$ be the map that sends $b_h$
to $b_{h+1}$ for $h<n$ and sends $b_n$ to $0,$ and let
$y:V\to V$ be the map that sends $b_h$
to $b_{h-1}$ for $h>1$ and sends $b_1$ to~$0.$

Note that for $p,q\geq 0,$ the map
$y^p$ annihilates the first $p$ of $b_1,\dots,b_n,$
while $x^q$ annihilates the last $q.$
(In this motivating sketch, it will help to think of $p$ and $q$ as
small compared with $n.)$
If we want to annihilate both the first $p$ and the
last $q$ basis elements, we can first apply $y^p$
to get rid of the former,
then $x^p$ to bring those that remain back to their original values,
and then $x^q$ to annihilate the last $q$ of them;
in other words, apply $x^{p+q}y^p.$
In addition to annihilating the first $p$ and the last $q$ of the
$\!b\!$'s, this leaves the $\!b\!$'s that
remain shifted $q$ steps to the right.
If we want to pull them back some distance to the left, we can
then apply $y^r$ for some $r.$
Here we may as well take $r\leq p+q,$
since if we used a larger value, this would kill some of the
elements at the left-hand end of our string, which we could
just as well have achieved by using a larger $p$ at the first step.
Thus, our combined operation has the
form $y^r x^{p+q} y^p$ with $p\geq 0,$ $q\geq 0,$ and $r\leq p+q.$

Clearly, we could have achieved the same result by {\em first}
annihilating the last $q$ of the $\!b\!$'s
using $x^q,$ then the first $p$ using $y^{p+q},$
and finally shifting those that remained using $x^{p+q-r}.$
Thus, $x^{p+q-r} y^{p+q} x^q = y^r x^{p+q} y^p.$
Renaming the exponents $p+q-r,$ $p+q$ and $q$ as
$i,$ $j$ and $k$ gives the first set of relations in the next result,
which records some basic properties of the monoid these relations
define.

\begin{lemma}\label{L.monoid}
The monoid $\F$ presented by two generators $x$ and $y,$ and
the infinite system of relations
\begin{equation}\begin{minipage}[c]{35pc}\label{d.many_rels}
$x^i y^j x^k\ =\ y^{j-i} x^j y^{j-k}$\quad
for all $i,$ $j,$ $k$ such that $0\leq i\leq j$ and $0\leq k\leq j$
\end{minipage}\end{equation}
can also be presented by the subset consisting of the relations
\begin{equation}\begin{minipage}[c]{35pc}\label{d.few_rels1}
$x\,y^j x^j\ =\ y^{j-1} x^j,$\quad $y^j x^j y\ =\ y^j x^{j-1},$\quad
for $1\leq j$
\end{minipage}\end{equation}
\textup{(}the cases of~\eqref{d.many_rels} where the $\!3\!$-tuple
$(i,j,k)$ has one of the forms $(1,j,j)$ or $(0,j,j\,{-}\,1),$
with the two sides of the relation interchanged in the
latter case\textup{)};
and, likewise, by the subset consisting of the relations
\begin{equation}\begin{minipage}[c]{35pc}\label{d.few_rels2}
$x^j y^j x\ =\ x^j y^{j-1},$\quad $y\,x^j y^j\ =\ x^{j-1} y^j,$ \quad
for $1\leq j$
\end{minipage}\end{equation}
\textup{(}the cases where
$(i,j,k)$ has one of the forms $(j,j,1),$ $(j\,{-}\,1,j,0),$
with the same interchange of sides in the latter case\textup{)}.

In $\F,$ every element is equal to the left-hand side
of~\eqref{d.many_rels} for a unique choice of $i,$ $j,$ $k$
satisfying the indicated inequalities; hence, equivalently, to the
right-hand side of~\eqref{d.many_rels} for the same $i,$ $j,$ $k.$

$\F$ embeds in the direct product of monoids
$\lang x,y\mid xy=1\rang\times\lang x,y\mid yx=1\rang$
via the homomorphism $x\mapsto(x,x),$ $y\mapsto(y,y).$
\end{lemma}

\begin{proof}
Since the set of relations~\eqref{d.few_rels1}
is a subset of~\eqref{d.many_rels},
to get our first assertion we must show that~\eqref{d.few_rels1}
implies all the relations in~\eqref{d.many_rels}.
Let us first show that it implies the family
\begin{equation}\begin{minipage}[c]{35pc}\label{d.tween_rels1}
$x^i y^j x^j\ =\ y^{j-i} x^j,$\quad $y^j x^j y^i\ =\ y^j x^{j-i},$
\quad for $0\leq i\leq j.$
\end{minipage}\end{equation}
Here the cases with $i=0$ are vacuous, and
those with $i=1$ are~\eqref{d.few_rels1}.
For $i>1,$ we note that the left-hand side of the first
equation of~\eqref{d.tween_rels1} can be written
$x^{i-1}(x\,y^j x^j),$ which by~\eqref{d.few_rels1} reduces to
$x^{i-1} y^{j-1} x^j.$
Writing this as $(x^{i-1} y^{j-1} x^{j-1})\,x,$
we may assume by induction on $i$ that this can be reduced
using~\eqref{d.few_rels1} to
$(y^{(j-1)-(i-1)} x^{j-1})\,x = y^{j-i} x^j,$
the desired expression.
The second equation of~\eqref{d.tween_rels1}
is obtained by the same calculation, with the order
of factors reversed and $x$ and $y$ interchanged.

To get the full set of
relations~\eqref{d.many_rels}, take any $i,$ $j,$ $k$
as in those relations, and note
that the expression $x^i y^j x^j y^{j-k}$ can
be reduced using~\eqref{d.tween_rels1} in two ways:
On the one hand,
$x^i\,(y^j x^j y^{j-k}) =x^i\,(y^j x^{j-(j-k)}) = x^i y^j x^k,$
which is the left-hand side of~\eqref{d.many_rels};
on the other hand, $(x^i y^j x^j)y^{j-k} =(y^{j-i} x^j) y^{j-k},$
which is the right-hand side.
So these are equal, as desired.

By left-right symmetry,~\eqref{d.many_rels} is similarly
equivalent to~\eqref{d.few_rels2}.

Let us show next that every $a\in \F$ can be represented by
an expression as in~\eqref{d.many_rels}.
Let $w$ be an expression for $a$ of minimal length
in $x$ and $y.$
We can write $w$ as an alternating product of nonempty ``blocks'' of
$\!x\!$'s and $\!y\!$'s.
If it consists of $\leq 2$ blocks, it is immediate that
it has one of the forms shown in~\eqref{d.many_rels}, so
assume it has at least $3$ blocks.
I claim that $w$ cannot consist of $\geq 4$ blocks.
For if it does, consider two adjacent blocks which are neither
the rightmost nor the leftmost pair.
By the right-left symmetry of the statements we are dealing with, we
may assume for simplicity that the length of the first of these blocks
is at least that of the second, and by symmetry in $x$ and $y,$
that the first block is a power of $x$
and the second a power of $y;$ so our product of two blocks can
be written $x^i y^j$ with $i\geq j\geq 1.$
Since this pair of blocks does not occur at the right
end of $w,$ it must be followed by an $x;$
so $w$ has a subword $x^i y^j x,$ which in turn has
the subword $x^j y^j x,$ which can be reduced
by~\eqref{d.few_rels2} to the shorter word
$x^j y^{j-1},$ contradicting our minimality assumption.
So $w$ must have just three blocks, and we can assume
without loss of generality that it has the form $x^i y^j x^k.$
If $j$ is $\geq$ both $i$ and $k,$ we are done.
If not, assume without loss of generality that $i>j.$
Then $w$ has the subword $x^j y^j x,$ which,
as above, leads to a contradiction to minimality.
Putting aside our ``w.l.o.g'' assumptions, what we have shown is
that a word $w$ of minimal length representing $a$ will
be of one of the forms shown in~\eqref{d.many_rels}.
Whichever of those forms
it has,~\eqref{d.many_rels} allows us to represent
$a$ in the other form as well (though that may not be
a minimal-length expression for it).

To show that the expression for $a$ in each of these forms is
unique, and simultaneously obtain the embedding of the final
assertion of our lemma, let us first note that the
generators $x$ and $y$ of each of the monoids
$\lang x,y\mid xy=1\rang$ and $\lang x,y\mid yx=1\rang$
clearly satisfy~\eqref{d.few_rels1}.
Hence we get a homomorphism
\begin{equation}\begin{minipage}[c]{35pc}\label{d.<>x<>}
$h:\F\ \to\ \lang x,y\mid xy=1\rang\ \times\ \lang x,y\mid yx=1\rang$
\end{minipage}\end{equation}
carrying $x$ to $(x,x)$ and $y$ to $(y,y).$
Now in $\lang x,y\mid xy=1\rang,$
the expressions $y^i x^j$ give a normal form
(as is easily checked by the method of \cite{<>}),
while in $\lang x,y\mid yx=1\rang,$ the same is true of the
expressions $x^i y^j.$
The image under $h$ of an element $x^i y^j x^k$
with $0\leq i\leq j$ and $0\leq k\leq j,$
written using these normal forms, is $(y^{j-i} x^k, x^i y^{j-k}).$
Here the first component determines $k,$
the second determines $i,$
and with these known, either component determines $j.$
So elements represented by distinct expressions $x^i y^j x^k$ as
in~\eqref{d.many_rels} have distinct images under $h,$ showing,
on the one hand, that~\eqref{d.<>x<>} is an embedding,
and on the other, that distinct expressions as in the left-hand
side of~\eqref{d.many_rels} represent distinct elements of~$\F.$
\end{proof}

Intuitively, the first relation of~\eqref{d.few_rels1}
says that after a factor $x^j y^{j-1},$ one can ``drop''
a factor $y x;$ in other words, that $x^j y^{j-1}$ behaves like $x$
in that respect.
The other relation of~\eqref{d.few_rels1}
and those of~\eqref{d.few_rels2} have the obvious analogous
interpretations.
More generally, I claim that
\begin{equation}\begin{minipage}[c]{35pc}\label{d.w_like_x}
If $w$ is a word in $x$ and $y$ which has strictly more $\!x\!$'s
than $\!y\!$'s, then in the monoid $\F$ we have $wyx=w=xyw.$
\end{minipage}\end{equation}
\begin{equation}\begin{minipage}[c]{35pc}\label{d.w_like_y}
If $w$ is a word in $x$ and $y$ which has strictly more $\!y\!$'s
than $\!x\!$'s, then in the monoid $\F$ we have $wxy=w=yxw.$
\end{minipage}\end{equation}

Indeed, by symmetry in $x$ and $y,$ it suffices to
prove~\eqref{d.w_like_x}, and by right-left symmetry
it suffices to prove the first relation thereof, $wyx=w.$
The last sentence of Lemma~\ref{L.monoid} shows that it suffices
to prove that relation in each of the monoids
$\lang x,y\mid xy=1\rang$ and $\lang x,y\mid yx=1\rang.$
The result is trivially true in the latter monoid.
In the former, $w$ can be reduced to a word $y^i x^j$ such
that $j-i=($number of $\!x\!$'s in $w)-($number of $\!y\!$'s in $w).$
This difference is positive by the hypothesis of~\eqref{d.w_like_x},
so $j>0;$ so $w$ is equal in $\lang x,y\mid xy=1\rang$
to a word ending in $x,$ so by the relation $xyx=x,$
we indeed have $wyx=w.$

\section{A quick introduction to inverse monoids}\label{S.inv_md}

The concept of {\em inverse monoid}, mentioned in the
Introduction, has as its motivating case
the set of {\em partial one-to-one maps}
of a set $S$ into itself; i.e., the one-to-one maps from
subsets of $S$ to~$S.$

Clearly, the composite of two partial one-to-one maps is again such
a map, as is the identity map of $S,$ so such maps form a monoid.
Further, if $x$ is such a partial map, and we view it as
a relation, hence a subset of $S\times S,$ then the inverse relation
$y=\{(t,s)\mid (s,t)\in x\}$ is again a partial one-to-one map.
If we consider all identities satisfied by such partial
maps, under the monoid operations and this ``inverse'' operation,
these determine a {\em variety} of monoids with an additional
unary operation.
We shall call objects of this variety {\em inverse monoids}.

Among semigroup-theorists, the ``inverse'' operation
is generally written $y=x^{-1}.$
But this would conflict badly with the usage of ring theory,
so, following~\cite[\S 4]{PA+KG}, I will instead write~$x^*.$

It turns out that (as with the inverse operation of a group) if
a monoid admits any unary operation $^*$ satisfying the identities
of the above variety, then that operation is unique
\cite[Theorem~1.17]{C+P1}, \cite[Theorem~1.3]{MVL}; so (like groups)
inverse monoids can be identified with a subclass of the monoids,
and they are in fact generally so described in the literature.
Precisely, it is easy to see that the identities of
inverse monoids include $x = x x^* x$ and $x^* = x^* x x^*,$
and the standard definition is that an inverse monoid
is a monoid $M$ such that
\begin{equation}\begin{minipage}[c]{35pc}\label{d.uniq}
For every $x\in M$ there is a {\em unique} $y\in M$
satisfying $xyx=x$ and $yxy=y$\\
\cite[p.\,6, conditions~(1) and~(2)]{MVL},
\cite[Definition~II.1.1]{MP},
\cite[p.\,28, lines~9-11]{C+P1}.
\end{minipage}\end{equation}

(Caveat: Given an element $x$ in an inverse monoid, an
element $y$ satisfying $xyx=x,$ but not necessarily
$yxy=y,$  will not, in general, be unique.
For instance, in the monoid of partial one-to-one endomaps of
a set $S,$ if $x$ is the empty partial map, every $y$
satisfies that relation.)

Actually, most of the literature in this
area describes its subject as
{\em inverse semigroups,} where no identity element is assumed.
Those that are monoids are indeed considered, but as a subcase.
However, as with the relationship between
ordinary semigroups and monoids, or nonunital and unital
rings, either version of the theory can be reduced to the
other, so which is treated as primary is largely a matter of
taste; and I will talk about {\em inverse monoids} here.

The most commonly used characterizations of inverse monoids
are not given by identities;~\eqref{d.uniq} is such
a characterization.
We recall two others:
if $M$ is a monoid and $^*$ a unary operation on
its underlying set, then each of the following conditions is
equivalent to $(M,^*)$ being an inverse monoid.
\begin{equation}\begin{minipage}[c]{35pc}\label{d.embed}
$(M,{}^*)$ can be embedded in the monoid of one-to-one endomaps
of some set $S,$ in such a way that ${}^*$ acts
as the relational inverse map
\cite[Corollary~IV.1.9]{MP},
\cite[p.\,36, Theorem~1.5.1]{MVL},
\cite[Theorem~1.2]{C+P1}.
\end{minipage}\end{equation}
\begin{equation}\begin{minipage}[c]{35pc}\label{d.idpts}
For every $x\in M,$ one has $x=x x^* x,$ and, further, every pair
of idempotent elements of $M$ commute
\cite[p.\,6, Theorem~3]{MVL},
\cite[Theorem~1.17(i)]{C+P1},
cf.\ \cite[Theorem~II.1.2]{MP}.
\end{minipage}\end{equation}

To move toward a characterization by identities,
note that the identities satisfied by partial
one-to-one operations on a set include
the conditions making $^*$ an {\em involution} of $M,$
\begin{equation}\begin{minipage}[c]{35pc}\label{d.invol}
$(ab)^*\ =\ b^* a^*$\quad and\quad $a^{**}\ =\ a.$
\end{minipage}\end{equation}
From these it follows that if an inverse monoid $(M,^*)$
is generated by a subset $X,$ then as
a monoid, $M$ is generated by $X\cup X^*.$
Indeed, the submonoid generated by that set will
contain $X$ and be closed under the monoid operations, and
by~\eqref{d.invol} it will be closed under~$^*,$
hence it will be all of~$M.$

Now the reason why~\eqref{d.idpts} is not a set of {\em identities}
is that the condition that idempotents
commute is not an equation in arbitrary elements of $M.$
However, given the identity
$x=x x^* x$ of~\eqref{d.idpts}, we see that if $e$ is an idempotent,
then $e=e e^* e = e (e^* e^*) e = (e e^*)(e^* e^{**}),$
a product of two idempotents of the form $x x^*.$
Hence if we combine the monoid identities with the identity
$x=x x^* x,$ and the identity saying that any two elements of
the forms $x x^*$ and $y y^*$ commute, then these together
imply~\eqref{d.idpts}, and so define the variety of inverse monoids.

A feature that inverse monoids share with groups, though
not nearly as easy to prove, is
\begin{equation}\begin{minipage}[c]{35pc}\label{d.hom}
If $(M,^*)$ is an inverse monoid, and $N$ is a homomorphic
image of $M$ as a monoid, then the operation $^*$ on $M$
induces an operation on $N,$ whence $N$
becomes an inverse monoid \cite[Theorem~7.36]{C+P2},
\cite[Lemma II.1.10]{MP}.
\end{minipage}\end{equation}

In view of the uniqueness of the inverse monoid structure on
a monoid when one exists, I will sometimes take the shortcut
of saying that a monoid ``is'' an inverse monoid (as in the
standard usage in the field) when I mean that it admits
an operation $^*$ making it an inverse monoid.

Alongside groups, and monoids of partial one-to-one maps
of sets, another important class of
inverse semigroups are the semilattices, where $^*$ is taken to be the
identity map.

Extensive developments of the theory of inverse semigroups can
be found in~\cite{MVL} and~\cite{MP}.

\section{$\F$ as a free inverse monoid}\label{S.F}

In~\S\ref{S.monoid}, we motivated the structure
of our monoid $\F$ in terms of certain endomorphisms
of finite-dimensional vector spaces.
What we did can be looked at as a partial-one-one-map
construction in a different guise.
Each of $x$ and $y,$ and hence also
every monoid word in those elements, sends certain
members of our basis $\{b_1,\dots,b_n\}$ to members of that basis,
and does so in a one-to-one
fashion, while it maps the remaining elements of that basis to $0.$
If we associate to every endomorphism $z$ of $V$ that acts in
such a way on our basis the partial endomap of that basis which
agrees with $z$ on elements that $z$ sends to basis
elements, and is undefined on those $z$ sends $0,$ then the
resulting partial maps compose like the given linear endomorphisms.

On each of our finite-dimensional vector spaces,
the monoid of endomorphisms generated by $x$ and $y$ is finite;
but if we think of $x$ and $y$
as acting simultaneously on $\!n\!$-dimensional spaces for
all natural numbers $n,$ we shall find that we
get the infinitely many distinct elements of $\F.$
We will deal with actions on vector spaces in the
next section; here, let us show monoid-theoretically that

\begin{lemma}\label{L.free}
The monoid $\F$ of Lemma~\ref{L.monoid} is the free
inverse monoid on one generator $x,$ with $y=x^*.$
\end{lemma}

\begin{proof}
To see that $\F$ is an inverse monoid, we shall use its
representation, established in Lemma~\ref{L.monoid},
as a subdirect product of
\begin{equation}\begin{minipage}[c]{35pc}\label{d.bicyc}
$\lang x,y\mid xy=1\rang$\quad and\quad $\lang x,y\mid yx=1\rang.$
\end{minipage}\end{equation}
These are two copies of the same monoid, called the
{\em bicyclic} monoid, which has a natural representation
by partial endomaps of the natural numbers,
with the left-invertible generator acting as the right shift
and the right-invertible generator
as the left shift (undefined at $0).$
The involution $^*$ of each of the inverse monoids
of~\eqref{d.bicyc} interchanges
the generators $x$ and $y$ of that monoid,
so the submonoid $\F$ generated by $(x,x)$ and $(y,y)$ is
closed under coordinatewise application
of $^*,$ hence is itself an inverse monoid.

To show that the inverse monoid $(\F,{}^*)$ is free on $\{x\},$
it suffices to show that for any element $x$
of an inverse monoid, if we write $x^*=y,$ then
$x$ and $y$ satisfy the relations of~\eqref{d.few_rels1}.
To get the first of these relations, let us write
$x y^j x^j$ as $(xy) (y^{j-1} x^{j-1}) x.$
Since $y=x^*,$ we see from~\eqref{d.invol}
that $(y^{j-1})^*=x^{j-1},$
hence both $xy$ and $y^{j-1} x^{j-1}$ are idempotent, hence
by~\eqref{d.idpts} they commute, so we can write
$(xy) (y^{j-1} x^{j-1}) x$ as $(y^{j-1} x^{j-1}) (xy) x =
y^{j-1} x^{j-1} (xyx) =
y^{j-1} x^{j-1} x = y^{j-1} x^j,$ giving the desired relation.
Applying the involution $^*,$ we get the other
relation of~\eqref{d.few_rels1}.
\end{proof}

The normal form for elements of the free
inverse monoid on one generator given by the expressions
on either side of~\eqref{d.many_rels}
is also obtained in \cite[proof of Lemma~4.1]{PA+KG}.
I do not know
whether a system of defining monoid relations as economical
as~\eqref{d.few_rels1} or~\eqref{d.few_rels2} has previously
been noted.

The multiplicative monoids of the rings we will be looking at in
this note
are not, in general, inverse monoids; but we will nonetheless
be interested in pairs of elements $x$ and $y$ of these rings
that satisfy the relations holding between elements $x$ and $x^*$
of an inverse monoid.
Let us therefore make

\begin{definition}\label{D.strong_inner}
If $x$ is an element of a monoid $M,$ we shall call an element
$y\in M$ a {\em strong inner inverse} to $x$ if the submonoid of $M$
generated by $x$ and $y$ can be made an inverse monoid with $y=x^*;$
equivalently, if there exists a monoid homomorphism $\F\to M$
carrying $x,y\in\F$ to the elements of $M$ denoted by these
same symbols; equivalently, if $x,y\in M$ satisfy any
of the equivalent systems of monoid relations~\eqref{d.many_rels},
\eqref{d.few_rels1},~\eqref{d.few_rels2}.
\end{definition}

Clearly, the relation ``is a strong inner inverse of'' is symmetric.

From our observations on the bicyclic monoid
$\lang x,y\mid xy=1\rang\cong\lang x,y\mid yx=1\rang,$
we see that any right or left inverse of an element $x$
of a monoid (and hence in particular, any $\!2\!$-sided
inverse) is a strong inner inverse to $x.$

If every element of a monoid $M$ has a
{\em unique} strong inner inverse, then $M$ will, {\em a fortiori},
satisfy~\eqref{d.uniq}, and so be an inverse monoid.
But we shall see below that in many rings,
every element has a strong inner inverse, without these being unique.

(We remark that our use of ``strong'' in ``strong inner inverse''
should not be confused with the use of ``strongly'' in the existing
concept of a ``strongly regular ring'', a regular ring in which
all idempotents are central.
That condition is much more restrictive
than the condition that every element have a strong inner inverse
in our sense.)

\section{Many vector space endomorphisms have strong inner inverses \dots}\label{S.strong_inner}

We shall show here that if $V$ is a vector space over
a division ring $D,$ then a large class of
elements $x\in\r{End}_D(V)$ have strong inner inverses.
Our proof will make use of the following weak version of Jordan
canonical form, which holds over division rings.

\begin{lemma}\label{L.Jordan}
Suppose $V$ is a finite-dimensional vector space over a
division ring $D,$ and $x$ a vector-space endomorphism of $V.$
Let us call a nonzero $\!x\!$-invariant subspace $W$ of $V$
``$\!x\!$-basic''
if it admits a basis $b_1,\dots,b_n$ such that $x(b_h)=b_{h+1}$
for $h<n,$ and $x(b_n)=0.$
Then $V$ can be written as the direct sum of a subspace on
which $x$ acts invertibly, and a family of $\!x\!$-basic subspaces.
\end{lemma}

\begin{proof}
By Fitting's Lemma \cite[Theorem~19.16]{TYL_1st},
there exists an $N\geq 1$ such that
\begin{equation}\begin{minipage}[c]{35pc}\label{d.V=I+K}
$V\ =\ \r{im}(x^N)\ \oplus\ \r{ker}(x^N).$
\end{minipage}\end{equation}
Clearly, each of these summands is $\!x\!$-invariant, and
$x$ acts injectively (hence invertibly) on the former,
since otherwise the summands would have nonzero intersection.
So it will suffice to show that $\ker(x^N)$ is a direct sum
of $\!x\!$-basic subspaces.
This can be done in a well-known manner which we now sketch.

Noting that
$\r{ker}(x)\cap\r{im}(x^N)\subseteq \ker(x^N)\cap \r{im}(x^N)=\{0\},$
we look at the chain of subspaces
\begin{equation}\begin{minipage}[c]{35pc}\label{d.ker_cap_im}
$\r{ker}(x)\cap\r{im}(x^{N-1})\ \subseteq
\ \r{ker}(x)\cap\r{im}(x^{N-2})\ \subseteq\ \dots\ \subseteq
\ \r{ker}(x)\cap\r{im}(x)\ \subseteq\ \r{ker}(x).$
\end{minipage}\end{equation}
We take a basis $B_{N-1}$ of the first of these,
extend it to a basis $B_{N-1}\cup B_{N-2}$ of the second,
and so on, getting a basis $B_{N-1}\cup\dots\cup B_{1}\cup B_{0}$
of $\r{ker}(x),$ with the $B_i$ disjoint.
For $0\leq i\leq N-1,$ since $B_i\subseteq\r{im}(x^i)$
we can write each $b\in B_i$ as $x^i(b')$
for some $b',$ which we see will still lie in $\r{ker}(x^N).$
If we consider the set
\begin{equation}\begin{minipage}[c]{35pc}\label{d.basis}
$\bigcup_{0\leq i\leq N-1,\ b\in B_i}\ \{x^h(b')\mid\ 0\leq h\leq i\},$
\end{minipage}\end{equation}
then this is easily shown to form a $\!D\!$-basis for $\r{ker}(x^N),$
and for each $i$ and each $b\in B_i,$
the span of the set $\{x^h(b')\mid\ 0\leq h\leq i\}$
is $\!x\!$-basic, since
$x$ carries $x^h(b')$ to $x^{h+1}(b')$ for $h\leq i,$
and $x^i(b')=b\in B_i\subseteq\r{ker}(x)$ to~$0.$
\end{proof}

We can now prove

\begin{theorem}\label{T.main}
Let $D$ be a division ring, $V$ a right vector space over $D,$
and $x$ a vector space endomorphism of $V$ such that $V$ can be written
as a direct sum of $\!x\!$-invariant subspaces, each of
which is either finite-dimensional, or has the property that
the action of $x$ on it is one-to-one, or has the property that
the action of $x$ on it is surjective.
Then $x$ has a strong inner inverse $y$ in $\r{End}_D(V),$
which carries each of those direct summands into itself.

If, moreover, our decomposition of $V$ includes,
for each positive integer $n,$ at least one
$\!n\!$-dimensional summand which is $\!x\!$-basic
in the sense of Lemma~\ref{L.Jordan},
then a strong inner inverse $y$ to $x$ can be chosen so that the
ring of additive-group endomorphisms of $V$ generated by the action
of $D$ and the actions of $x$ and $y$ has precisely the structure of
the monoid ring $D\,\F.$
In particular, if $D$ is a field $k,$ this says that the
$\!k\!$-algebra of endomorphisms of $V$ generated by
$x$ and $y$ is isomorphic to the monoid algebra $k\,\F.$
\end{theorem}

\begin{proof}
In proving the assertion of the first paragraph
we may, by the preceding lemma,
assume that each of the given $\!x\!$-invariant summands of $V$ either
has $x$ acting surjectively on it, or
has $x$ acting injectively on it, or is
$\!x\!$-basic; so we are reduced to proving that $x$ has
a strong inner inverse if $x$ acts in one of these
three ways on $V$ itself.

If $x$ is surjective or injective, then it is right
or left invertible in $\r{End}(V),$ and a right or
left inverse will be the desired strong inner inverse.

So suppose $V$ is $\!x\!$-basic, and let $\{b_1,\dots,b_n\}$ be
a basis of $V$ such that $x$ carries
each $b_h$ with $h<n$ to $b_{h+1},$ and carries $b_n$ to $0.$
As in the motivating sketch of~\S\ref{S.monoid},
define $y$ to carry $b_h$ to $b_{h-1}$ for $h>1,$ and $b_1$ to $0.$
We could verify that $y$ is a strong inner inverse to $x$
using the correspondence between one-to-one partial maps on our basis
of $V$ and certain endomorphisms of $V,$ as noted
at the beginning of~\S\ref{S.F}, but the hands-on proof that $x$
and $y$ satisfy the relations of~\eqref{d.few_rels1} is
quick enough, so I will sketch it.

To prove, first, the relation $x\,y^j x^j = y^{j-1} x^j$ for $1\leq j,$
let us compare the effect of these two monomials on some $b_h.$
If $j+h>n,$ then the $x^j$ at the right end of each monomial annihilates
$b_h,$ so the two sides of the relation indeed agree on $b_h.$
If $j+h\leq n$ (so in particular, $h<n),$ it is immediate to check
that each side gives $b_{h+1},$ and the relation again holds.

Similarly, on checking the results of applying the two sides of the
relation $y^j x^j y = y^j x^{j-1}$ for $1\leq j$ to $b_h,$ we
find that if $h=1$ or $h+(j-1)>n,$ both sides give $0,$ while in the
contrary case, both sides give $b_{h-1}.$
This completes the proof of the assertion of the first
paragraph of the theorem.

To get the assertion of the second paragraph, let $y$ be constructed
on each $\!x\!$-basic summand of our decomposition of $V$ as in the
proof of the first assertion.
Since $y$ is a strong inner inverse to $x,$ the
elements $x$ and $y$ satisfy the relations defining $\F,$
and since they are $\!D\!$-linear, they commute with the
action of the elements of $D;$ hence the actions of $x,$
$y$ and the elements of $D$ yield a $\!D\F\!$-module structure on $V.$
It remains to show that this module is faithful,
so let $a\in D\,\F-\{0\},$ and let us show that the
action of $a$ on $V$ is nonzero.

To this end, let $x^i y^j x^k,$ with $i,\,j,\,k$ as
in~\eqref{d.many_rels}, be the element of $\F$
having nonzero $\!D\!$-coefficient
in $a$ which gives the {\em least} $\!3\!$-tuple $(k,j,i)$
under lexicographic order.
(Note the reversed order of indices.)
Choose an $\!x\!$-basic subspace $V'$ in our given decomposition of
$V$ which has dimension exactly $j+1,$ and let $\{b_1,\dots,b_{j+1}\}$
be a basis for $V'$ on which $x$ and $y$ act as right and left
shift operators.

I claim that the element $b_{j-k+1}$ is not annihilated by $a.$
Consider first how $x^i y^j x^k$ acts on $b_{j-k+1}.$
The factor $x^k$ carries it to $b_{j+1}$ (note that any
higher power of $x$ would kill it), $y^j$ carries this to
$b_1$ (here any higher power of $y$ would kill it),
and $x^i$ brings this to $b_{i+1}.$
So it will suffice to show that none of the other terms occurring
in $a$ carry $b_{j-k+1}$ to an expression in which $b_{i+1}$ appears.
Let $x^{i'} y^{j'} x^{k'}$ be any other term occurring in $a.$

By our minimality assumption on $(k,j,i),$ we must
have $k'\geq k.$
If this inequality is strict, then $b_{j-k+1}$ is
killed on applying $x^{k'}$ to it.
On the other hand, if $k'=k,$ then $j'\geq j,$
and if this inequality is strict, our basis
element is killed on applying $y^{j'} x^k.$
Finally, if $k'=k$ and $j'=j,$ we must have $i'\neq i,$
and we see that $x^{i'} y^j x^k$ will carry
$b_{j-k+1}$ to $b_{i'+1}\neq b_{i+1}.$
Hence in $a(b_{j-k+1}),$
only the term $x^i y^j x^k(b_{j-k+1})$ contributes to the coefficient
of $b_{i+1},$ so $a(b_{j-k+1})\neq 0,$ so
$a$ indeed acts nontrivially on $V.$
\end{proof}

We shall see in the next section that in the first assertion of
the above theorem, the restriction to endomorphisms $x$ such
that $V$ is an appropriate direct sum of
sorts of $\!x\!$-invariant subspaces cannot be dropped.
Here is a case where that condition holds automatically.

\begin{corollary}\label{C.sum_fd}
Let $D$ be a division ring, and $R$ the direct product
of the endomorphism rings of a family of finite-dimensional
$\!D\!$-vector-spaces, $R=\prod_{i\in I} \r{End}(V_i)$
\textup{(}i.e., a direct product of
full matrix rings of various sizes over $D).$
Then every element of $R$ has a strong inner
inverse under multiplication.
\end{corollary}

\begin{proof}
Apply the preceding theorem with $V=\bigoplus V_i,$
and $R$ identified with the ring of endomorphisms of $V$ that
carry each $V_i$ into itself.
\end{proof}

We mentioned earlier that for $M$ a monoid,
the statement that every element
of $M$ has a strong inner inverse does not entail
that such strong inner inverses are unique.
For an explicit example, consider the monoid $\r{End}(V),$
where $V$ is a vector space with basis $\{b_1,\dots,b_n\}$
for some $n>1,$ and $x$ and $y$ again act by right and left shifts.
Then $x$ and $y$ are nilpotent, and do not commute.
Since $x$ is nilpotent, $1+x$ is invertible, hence conjugation
by that element is an automorphism of $\r{End}(V_i)$ which
fixes $x$ but not $y.$
Hence $(1+x)^{-1} y\,(1+x)$ is a strong inner inverse
to $x$ distinct from $y.$

One may ask

\begin{question}\label{Q.x_1...}
Given endomorphisms $x_1,\dots,x_r$ of a vector space $V$
over a field, or more generally, over a division ring,
under what natural conditions can we find
strong inner inverses $y_1,\dots,y_r$ to these elements, such that
the monoid generated by $x_1,\dots,x_r$ and $y_1,\dots,y_r$
is an inverse monoid?
\end{question}

If $r>1,$ such strong inner inverses need not exist for general
$x_1,\dots,x_r\in\r{End}(V),$ even if $V$ is finite-dimensional.
For instance, if we take $x_1$ and $x_2$
to be noncommuting idempotents, no submonoid
of $\r{End}(V)$ containing them satisfies~\eqref{d.idpts}.
A sufficient condition is that there exist a basis $B$ for $V$
such that every $x_i$ carries each member of $B$ either to
another member of $B,$ or to $0,$ and acts in a one-to-one
fashion on those that it does not send to $0.$
But this is not necessary.
For example, if we take for $x_1,\dots,x_r$ any automorphisms of $V,$
then they generate a group, which is an inverse monoid; but if any
of them has determinant $\neq\pm 1,$ it cannot permute a basis of $V.$

(I will mention one result that has a vaguely related feel.
Given a vector space $V$ and a finite family of subspaces
$S_1,\dots,S_n,$ one may ask under what conditions there
exists a basis $B$ for $V$ such that each $S_i$ is the
subspace spanned by a subset $B_i\subseteq B.$
By \cite[Exercise~6.1:16]{245}, such
a basis exists if and only if the lattice of subspaces
generated by $S_1,\dots,S_n,$ is distributive.)

There is an interesting analog,
for families of endomorphisms $x_1,\dots,x_r$ of vector spaces,
of the class of $\!x\!$-basic spaces.
Namely, for every finite connected subgraph $S$ of the Cayley
graph of the free group on $r$ generators $g_1,\dots,g_r,$ let $V_S$ be
a vector space with a basis $\{b_s\}$ indexed by the vertices $s$
of $S,$ and for $i=1,\dots,r,$ let $x_i$ act on $V_S$ by taking
$b_s$ to $b_t$ if $S$ has a directed edge from $s$ to $t$
indexed by $g_i,$ or to $0$ if no edge
indexed by $g_i$ comes out of $s;$
and likewise, let $x^*_i$ act by taking $b_t$ back to $b_s$ in cases
of the first sort, while taking $b_t$ to $0$ if $t$ has no
edge indexed by $g_i$ coming into it.
From a description of the free inverse monoid on $r$ generators due
to Munn \cite{WDM_trees} \cite[\S VIII.3]{MP}, \cite[\S 6.4]{MVL},
one finds that the algebra of operations on $\bigoplus_S V_S$
generated by the resulting maps $x_i$ and $x^*_i$ is isomorphic to
the monoid algebra of the free inverse monoid on $r$ generators.

\section{\dots but some do not}\label{S.ceg}

Let us show that in Theorem~\ref{T.main}, the hypothesis that
$V$ have a decomposition as a direct sum
of well-behaved $\!x\!$-invariant subspaces cannot be dropped.
We will use the next result, which concerns
monoids of ordinary (everywhere defined, not
necessarily one-to-one) endomaps of sets.

\begin{lemma}\label{L.bigcap}
Let $x$ be an endomap of a set $S.$
Then a necessary condition for $x$ to have a strong inner inverse
in the monoid of all endomaps of $S$ is
\begin{equation}\begin{minipage}[c]{35pc}\label{d.bigcap}
$x\,(\,\bigcap_{n\geq 0}\,x^n(S))\ =\ \bigcap_{n\geq 0}\,x^n(S).$
\end{minipage}\end{equation}
\end{lemma}

\begin{proof}
Suppose $x$ has a strong inner inverse $y.$
Since in~\eqref{d.bigcap}, the relation ``$\subseteq$'' clearly
holds, we must show the reverse inclusion.

So consider any $s\in \bigcap x^n(S).$
In view of the relations $x^n = x^n y^n x^n,$ we have
\begin{equation}\begin{minipage}[c]{35pc}\label{d.s=xnyns}
$s\ =\ x^n y^n s$ \ for all $n\geq 0.$
\end{minipage}\end{equation}
If we take $n\geq 1,$ apply $y$ to both sides of this equation,
and then apply to
the right-hand side the second relation of~\eqref{d.few_rels2}
with $n$ for $j,$ we get $y s = x^{n-1} y^n s.$
Hence $ys\in \bigcap x^n(S);$ and applying $x$ to both sides of
this relation, and invoking the $n=1$ case of~\eqref{d.s=xnyns},
we conclude that $s\in x\,(\,\bigcap x^n(S)),$ as desired.
\end{proof}

We shall see in the next section
that the condition of the above lemma is sufficient as well as
necessary; but we only need necessity for the example below.

That example will be obtained by slightly tweaking
the well-behaved example, implicit in Theorem~\ref{T.main},
of a space $V$ which is a direct sum of $\!x\!$-basic subspaces
of all natural number dimensions.

\begin{proposition}\label{P.ceg}
Let $V$ be a vector space over a
division ring $D,$ with a basis consisting
of elements $b_{n,i}$ for all positive integers $n$ and $i$
with $i\leq n,$
and one more basis element, $b_+;$ and let $x\in\r{End}(V)$ be given by
\begin{equation}\begin{minipage}[c]{35pc}\label{d.ceg_x}
$x(b_{n,i})=b_{n,i+1}$ if $i<n,$\quad
$x(b_{n,n})=b_+$ for all $n,$\quad $x(b_+)=0.$
\end{minipage}\end{equation}
Then $x$ has no strong inner inverse in $\r{End}(V).$
\textup{(}In fact, it has no strong inner inverse in the monoid
of all {\em set}-maps $V\to V.)$
\end{proposition}

\begin{proof}
Clearly, $\bigcap_{n\geq 0} x^n(V)$
is the $\!1\!$-dimensional subspace of $V$ spanned by $b_+.$
The image of this subspace under $x$ is the zero subspace,
so~\eqref{d.bigcap} is not satisfied, hence Lemma~\ref{L.bigcap}
gives the desired conclusion.
\end{proof}

The next result gets further mileage out of the above example.
The first assertion of that proposition
answers a question posed in an earlier version of~\cite{PN+JvS};
the second shows that in Lemma~\ref{L.bigcap} and
Proposition~\ref{P.ceg}, the relations characterizing
a strong inner inverse $y$ to $x$ cannot be replaced by the subset
consisting of the relations
$x^n y^n x^n = x^n$ and $y^n x^n y^n = y^n$ for all $n\geq 0.$

\begin{proposition}\label{P.unit_reg}
Let $V$ and $x$ be as in Proposition~\ref{P.ceg}.
Let $V_n,$ for each $n\geq 1,$ be the
subspace of $V$ spanned by $b_{n,1},\dots,b_{n,n},$
let $V_+$ be the subspace spanned by $b_+,$
let $R_0$ be the ring of all endomorphisms of $V$ that carry each
of these subspaces into itself, and let $L$ be the space of
endomorphisms of $V$ of finite rank.

Then $R_0 + L$ is a {\em unit regular} ring $R$ containing
$x,$ but containing no $y$ that satisfies
$x^n y^n x^n = x^n$ for all~$n.$

On the other hand, the full ring $\r{End}(V)$ contains an
element $y$ which satisfies both $x^n y^n x^n = x^n$ and
$y^n x^n y^n = y^n$ for all $n\geq 0,$
though by Proposition~\ref{P.ceg}, $y$ is not a strong inner
inverse to $x.$
\end{proposition}

\begin{proof}
Let us prove the above claims in reverse
order: first (easiest) the existence of a $y$ as
in the final sentence, then the non-existence result of the preceding
sentence, and finally, the unit-regularity of $R=R_0+L.$

A $y$ as in the final sentence of the proposition is defined by the
familiar formulas $y(b_{n,i})=b_{n,i-1}$ for $i>1,$ together with
the unexpected formulas $y(b_+)=b_{1,1}$ and
$y(b_{n,1})=b_{n+1,1}$ $(n\geq 1).$
In checking that for every $n$ we have $x^n y^n x^n = x^n,$
let us think of that relation as saying that $x^n y^n$ fixes
all elements of $x^n(V).$
Now $x^n(V)$ is spanned by the elements $b_{m,i}$ with $i>n,$ and $b_+;$
and it is straightforward to check
that elements of each of these sorts are fixed by $x^n y^n.$
Similarly, the desired relation $y^n x^n y^n = y^n$ says that
all elements of $y^n(V)$ are fixed by $y^n x^n.$
Now when $y$ is applied to a basis element $b_{m,i},$ we see that
the difference $m-i$ is increased by $1,$ whether we
are in the case $i>1$ or $i=1;$ so the result of applying
$y^n$ to a basis element $b_{m,i}$ is a basis element of the
form $b_{m',i'}$ with $m'-i'\geq n.$
It is immediate to check that any basis element with
this property is fixed by $y^n x^n.$
On the other hand, in evaluating $y^n x^n y^n(b_+),$ we can
use the fact that $b_+\in x^n(V),$ so by our observations
on the relations $x^n y^n x^n = x^n,$
$b_+$ is fixed under $x^n y^n.$
Now applying $y^n,$ we get $y^n x^n y^n(b_+)=y^n(b_+),$ as desired.

We turn next to $R = R_0 + L.$
It is clear that $R$ is a ring and $L$ an ideal of $R.$
The endomorphism of $V$ that carries
all basis elements of the form $b_{n,n}$ to $b_+,$ and
all other basis elements to zero, has rank $1,$ hence belongs to $L,$
and if we subtract it from $x$ we get a member of $R_0;$ so $x\in R.$

To prove the nonexistence of a $y\in R$ satisfying the
relations $x^n y^n x^n = x^n,$
consider any $y=y_0+y_1\in R,$ where $y_0\in R_0$ and $y_1\in L.$
Since $y_1$ has finite rank, its range lies in the sum of
$V_+$ and finitely many of the $V_n.$
Also, being a member of $V,$ the element $y(b_+)$ lies in the sum of
$V_+$ and finitely many of the $V_n.$
Hence we can choose an $n_0>0$ such that both the space
$y_1(V)$ and the element $y(b_+)$ have zero projection in
all of the spaces $V_n$ for $n\geq n_0.$

Now consider the element $x^{n_0} y^{n_0} x^{n_0}(b_{n_0,1}).$
The $x^{n_0}$ that acts on $b_{n_0,1}$ carries it to $b_+,$ and
subsequent iterations of $y$ will, by our choice of $n_0,$
carry this into a member of $V_+ + \sum_{n<n_0} V_n.$
But every member of this sum is annihilated by $x^{n_0};$
so $x^{n_0} y^{n_0} x^{n_0}(b_{n_0,1})=0,$
though $x^{n_0} (b_{n_0,1})=b_+.$
Hence $x^{n_0} y^{n_0} x^{n_0}\neq x^{n_0},$ as claimed.

It remains to show that $R$ is unit regular; i.e., that every
element has an inner inverse which is a unit.
We shall recall sufficient conditions for this to hold,
given in \cite[Lemma~3.5, p.\,600]{GB}
for a general ring $R$ with an ideal $L,$
in terms of properties of $R,$ $R/L$ and $L,$
and verify that these hold in the case at hand.
(I am indebted to Ken Goodearl and Pace Nielsen for supplying the
tools for this part of the proof.)

First, we must know that $R$ is regular.
This follows from \cite[Lemma~1.3]{KG}, since $L,$
being an ideal of $\r{End}(V),$ is regular,
and $R/L\cong R_0/(R_0\cap L),$ a homomorphic image of
the regular ring $R_0,$ is also regular.

Next, we must know that $eLe$ is unit regular for every
idempotent $e\in L.$
But $eLe\cong\r{End}(eV),$ the endomorphism ring of
a finite dimensional vector space, and so is
unit regular by \cite[Theorem~4.1(a)$\iff$(c)]{KG}.

We must also know that $R/L$ is unit regular.
For this we again use the fact
that $R/L\cong R_0/(R_0\cap L),$ this time combined with the fact
that $R_0,$ a direct product of unit regular rings,
is unit regular.

Finally, we need to know that every unit of $R/L$ lifts to
a unit of $R.$
Now a unit of $R/L$ and its inverse will be images of
elements $u,v\in R_0.$
Hence $uv-1$ and $vu-1$ will both lie in $R_0\cap L,$ hence
will be members of $R_0$ whose components equal zero in all but
finitely many of the algebras $\r{End}(V_n)$ and $\r{End}(V_+).$
Hence if we modify $u$ and $v$ only in their behavior on the
finitely many factors where they are not inverse to one another,
replacing their actions there with, say, the identity elements
of those factors, we get elements $u', v'\in R_0$ which {\em are}
inverse to one another, and these give the desired liftings to $R$
of the original unit and its inverse in $R/L.$
\end{proof}

In contrast to the first assertion of the above proposition, we know
from \cite[Theorem~4.8]{PN+JvS}, mentioned in~\S\ref{S.bkgd}, that
the $x$ of the above example has inner inverses in $R$ which satisfy
any {\em finite} subset of the relations~\eqref{d.many_rels}.

The referee has asked whether in a regular ring, a
product of two elements with strong
inner inverses must have a strong inner inverse.
Let us show that, in fact, the product of an element having a strong
inner inverse and an invertible element can fail to have a strong
inner inverse.
Clearly, this is equivalent to saying that the product of an
element $x$ {\em not} having a strong inner inverse and an
invertible element $u$
can have a strong inner inverse, and that is the form in which
it will be convenient to describe the example.

Let $V$ and $x\in\r{End}(V)$ again be as in Proposition~\ref{P.ceg},
and let $u$ be the automorphism of
$V$ which acts on our basis by the following permutation:
\begin{equation}\begin{minipage}[c]{35pc}\label{d.unit}
$u(b_{n,i})=b_{n,i-1}$ for $1<i\leq n,$\quad
$u(b_{n,1})=b_{n,n}$ for all $n,$\quad
$u(b_+)=b_+.$
\end{minipage}\end{equation}
Then $xu$ fixes all elements $b_{n,i}$ with $i>1,$ carries
all elements $b_{n,1}$ to $b_+,$ and annihilates that element.
We know from Proposition~\ref{P.ceg} that $x$ does not have
a strong inner inverse; however $xu$ does, by
Theorem~\ref{T.main}, using the decomposition of $V$ as the direct sum
of the one-dimensional subspaces with bases $\{b_{n,i}\}$ $(1<i\leq n),$
the $\!2\!$-dimensional subspace with basis $\{b_{1,1},\,b_+\},$
and the one-dimensional subspaces with bases
$\{b_{n+1,1}-b_{n,1}\}$ $(n\geq 1).$
Each of these subspaces is easily seen to be $\!xu\!$-invariant,
hence that theorem is indeed applicable.

(J.\,\v{S}ter (personal communication) has pointed out a quicker
way to see that the product of the above $x$ with some invertible
element has a strong inner inverse, though it does not
give that invertible element as explicitly:
By Proposition~\ref{P.unit_reg}, $x$ has an invertible
inner inverse $u'$ in $R_0+I\subseteq\r{End}(V).$
This makes $xu'$ idempotent, and hence its own strong inner inverse.)

The referee has also raised the following question.
Recall that for $R$ a ring, its {\em Pierce stalks} (the
stalks of the {\em Pierce sheaf} of $R)$ are the
factor-rings $R/I,$ as $I$ ranges over the maximal members of
the partially ordered set of proper ideals of $R$ generated by families
of central idempotents \cite{Pierce}, \cite[p.\,354]{B+B+R}.
Any ring $R$ is a subdirect product of its Pierce stalks.

\begin{question}\label{Q.Pierce}
If $x$ is an element of a regular ring $R,$ and the image
of $x$ in every Pierce stalk $R/I$ has a strong inner inverse,
must $x$ have a strong inner inverse in $R$?
\end{question}

If the answer is negative, a counterexample might look
something like the following.
Start with simple regular rings $R_i$ $(i\geq 1)$ in which
all elements have strong inner inverses (for instance, full
matrix rings over fields), and let $R\subseteq\prod R_i$ be a subdirect
product of the $R_i$ whose only central idempotents are
the elements of $\prod R_i$ with finitely many components $1$
and all other components $0,$
and those with finitely many components $0$
and all other components $1.$
Then the Pierce stalks of $R$ will be the rings $R_i,$ and
one stalk ``at infinity'', $R_\infty.$
Suppose now that $x$ and $y$ are elements of $R$
such that, of the relations~\eqref{d.few_rels1},
all but one hold between their images $x_1,\,y_1\in R_1,$
all but another hold between their
images $x_2,\,y_2\in R_2,$ and so on.
Then all these relations hold between
their images $x_\infty,\,y_\infty\in R_\infty,$ while in each
of the rings $R_i,$ the image $x_i$ of $x$ has, by our
hypothesis on those rings, {\em some} strong inner inverse.
So the hypotheses of Question~\ref{Q.Pierce} are satisfied.
However, there is no evident reason why it should
be possible to modify all the $y_i$ so as to get
a strong inner inverse to $x$ which still lies in the
chosen subring $R\subseteq\prod R_i.$
On the other hand, it is not clear how one might come up with
an $R$ in which such an inner inverse was guaranteed not to exist.

Even if the above idea of what a counterexample could look
like is roughly correct, it may be naive to assume that
one could have only {\em one} of the equations~\eqref{d.few_rels1}
fail in each $R_i.$
This suggests

\begin{question}\label{Q.sets_of_rels}
For which subsets $S$ of the set of equations~\eqref{d.few_rels1}
is it the case that there exist a ring $R$ and elements $x,y\in R$
which satisfy all the relations in $S,$ but
none of the other relations of~\eqref{d.few_rels1}?
\end{question}

\section{Digression: A characterization of set-maps having strong inner inverses}\label{S.actions}

Here is the promised strengthening of Lemma~\ref{L.bigcap}.
(It will not be used in subsequent sections.)

\begin{theorem}\label{T.bigcap}
Let $x$ be an endomap of a set $S.$
Then condition~\eqref{d.bigcap} is necessary and sufficient for
$x$ to have a strong inner inverse $y$ in the monoid of
all endomaps of $S.$
\end{theorem}

\begin{proof}
In view of Lemma~\ref{L.bigcap}, we only have to prove
sufficiency, so assume $x$
satisfies~\eqref{d.bigcap}, and let us construct~$y.$

For every $s\in S,$ we define its ``depth'',
\begin{equation}\begin{minipage}[c]{35pc}\label{d.depth}
$d(s)\ =$ greatest integer $n\geq 0$ such
that $s\in x^n(S)$ if this exists,
or $\infty$ if $s\in\bigcap_{n\geq 0} x^n(S).$\\
(In statements such as~\eqref{d.dxs} and~\eqref{d.ds-1}
below, we will understand $\infty+1=\infty=\infty-1.)$
\end{minipage}\end{equation}
In particular, $d(s)=0$ if and only if $s\notin x(S).$
Clearly,
\begin{equation}\begin{minipage}[c]{35pc}\label{d.dxs}
For all $s\in S,$ we have $d(x(s))\geq d(s)+1.$
\end{minipage}\end{equation}
Moreover, I claim that
\begin{equation}\begin{minipage}[c]{35pc}\label{d.ds-1}
If $d(s)>0,$ then
$s$ can be written as $x(t)$ for some $t$ with $d(t) = d(s)-1.$
\end{minipage}\end{equation}
Indeed, for $0<d(s)<\infty,$~\eqref{d.ds-1} is straightforward,
while the case where $d(s)=\infty$ is our assumption~\eqref{d.bigcap}.

We will begin the construction of $y$ by defining it on all $s\in x(S).$
In this case, the criterion for choosing $y(s)$ will be fairly natural,
but with one nonobvious restriction, applying to elements $s\in S$
that satisfy
\begin{equation}\begin{minipage}[c]{35pc}\label{d.pre-stable}
$d(x^n(s))\ =\ d(s)+n$ \quad for all $n\geq 0$ (cf.~\eqref{d.dxs}).
\end{minipage}\end{equation}
For such an $s,$ {\em if} the function $y$ that we shall define below
satisfies
\begin{equation}\begin{minipage}[c]{35pc}\label{d.stable}
$y\,x^n(s)\ =\ x^{n-1}(s)$ \quad for all $n \geq 1,$
\end{minipage}\end{equation}
let us call $s$ {\em $\!y\!$-stable}.
We now specify $y$ on $x(S)$ by the rule
\begin{equation}\begin{minipage}[c]{35pc}\label{d.y_on_xS}
For each $s\in x(S),$ let $y(s)$ be an
element $t$ such that $s = x(t)$ and $d(t) = d(s)-1$
(as in~\eqref{d.ds-1}).
Moreover, if there are any $s\in S$
satisfying~\eqref{d.pre-stable}, then
we make our choice of $y$ satisfy~\eqref{d.stable}
for {\em at least one} such $s$
(i.e., we make at least one such $s$ $\!y\!$-stable).
\end{minipage}\end{equation}

Can we achieve the second condition of~\eqref{d.y_on_xS}?
If $x$ is one-to-one on the $\!x\!$-orbit of some $s$
satisfying~\eqref{d.pre-stable},
we can clearly define $y$ on elements $x^n(s)$ by~\eqref{d.stable}.
If, on the other hand,
that orbit must eventually become periodic,
and taking a new $s$ in the periodic part,
$x$ will be one-to-one on the orbit of that element,
and we can define $y$ on that orbit by~\eqref{d.stable}.
So the second sentence of~\eqref{d.y_on_xS} can indeed be achieved.

Note that, in view of~\eqref{d.dxs}, and the condition
$d(t)=d(s)-1$ in~\eqref{d.y_on_xS}, we can say that
\begin{equation}\begin{minipage}[c]{35pc}\label{d.ynxn}
Once $y$ has been defined on $x(S)$ so as to satisfy~\eqref{d.y_on_xS},
we can use this partial definition of $y$
to evaluate $y^n(s)$ for any $s\in S$ with $d(s)\geq n.$
\end{minipage}\end{equation}

We now want to define $y$ on elements $s\notin x(S).$

A fairly easy case is that in which $d$ makes a ``jump'' somewhere
on the $\!x\!$-orbit of $s$:
\begin{equation}\begin{minipage}[c]{35pc}\label{d.d>n}
If $s\in S$ with $d(s)=0,$ and if
for some positive integer $n,$ $d(x^n(s)) > n,$
then letting $n$ be the least value for which this is true,
we define $y(s)=y^{n+1} x^n (s).$
\end{minipage}\end{equation}

In view of the hypothesis of~\eqref{d.d>n}
that $d(x^n(s))>n,$ we see by~\eqref{d.ynxn}
that $y^{n+1} x^n(s)$ is determined by the partial
definition of $y$ that we have so far; so~\eqref{d.d>n} makes sense.

Finally, let us define $y$ on those $s\notin x(S)$ such that
$d(x^n(s)) = n$ for all $n\geq 0.$
We will need the following equivalence relation.
\begin{equation}\begin{minipage}[c]{35pc}\label{d.approx}
For $s,s'\in S$ such that
$d(x^n(s)) = n = d(x^n(s'))$ for all $n\geq 0,$
we shall write $s\approx s'$
if there exists some $n$ such that $x^n(s) = x^n(s').$
\end{minipage}\end{equation}

We now specify
\begin{equation}\begin{minipage}[c]{35pc}\label{d.if_d=n}
If $s\in S$ satisfies $d(x^n(s)) = n$ for all $n\geq 0,$
then we take $y(s)$ to be any $\!y\!$-stable element
(any element satisfying~\eqref{d.stable}), subject only to
the condition that if $s\approx s',$ then $y(s)=y(s').$
\end{minipage}\end{equation}

To see that we can do this, note that the situation
``$d(x^n(s)) = n$ for all $n\geq 0$'' is a case
of~\eqref{d.pre-stable},
hence by the second sentence of~\eqref{d.y_on_xS}, if there
exist $s$ as in~\eqref{d.if_d=n}, then there also exist $\!y\!$-stable
elements, which is all we need to carry out~\eqref{d.if_d=n}.

This completes the construction of $y.$
Let us record two immediate consequences of the
first sentence of~\eqref{d.y_on_xS}.
First, the fact that $y(s)$ is chosen
to be a preimage of $s$ under $x$ can be reworded:
\begin{equation}\begin{minipage}[c]{35pc}\label{d.xys}
If $d(s)>0,$ then $xy(s)=s.$
\end{minipage}\end{equation}
Second, combining the condition on $d(t)$
in~\eqref{d.y_on_xS} with the fact that for $s\notin xS,$
$d(y(s))\geq 0> d(s)-1,$ we see that
\begin{equation}\begin{minipage}[c]{35pc}\label{d.dys}
For all $s\in S,$ we have $d(y(s))\geq d(s)-1$ (cf.~\eqref{d.dxs}).
\end{minipage}\end{equation}

Let us now show that $x$ and the
$y$ we have constructed satisfy the relations~\eqref{d.few_rels1}.
The first of these, $x y^j x^j = y^{j-1} x^j,$
says that for every $s\in S,$ the element $y^{j-1} x^j(s)$ is fixed
under $xy.$
This is easy: by~\eqref{d.dxs} and~\eqref{d.dys}, every $s\in S$
satisfies $d(y^{j-1} x^j(s)) \geq 1,$ so by~\eqref{d.xys},
$y^{j-1} x^j(s)$ is indeed fixed under~$xy.$

The other relation from~\eqref{d.few_rels1},
\begin{equation}\begin{minipage}[c]{35pc}\label{d.yjxjy=}
$y^j x^j y\ =\ y^j x^{j-1},$
\end{minipage}\end{equation}
automatically holds when applied to elements $s$ with $d(s) > 0,$
since the two sides differ by a right factor of $xy,$ again
allowing us to apply~\eqref{d.xys}.

So suppose $d(s) = 0.$

Assume first that $s$ has the property that
$d(x^n(s)) > n$ for some $n,$
and as in~\eqref{d.d>n}, let $n\geq 1$ be the least such value.
Applying the two sides of~\eqref{d.yjxjy=} to $s,$
and using~\eqref{d.d>n} on the left-hand side, we see
that we need to prove
\begin{equation}\begin{minipage}[c]{35pc}\label{d.need}
$y^j x^j y^{n+1} x^n(s)\ =\ y^j x^{j-1}(s).$
\end{minipage}\end{equation}
Now because we have assumed the element $s$ satisfies
$d(x^n(s))\geq n+1,$ we see that on the left-hand
side of~\eqref{d.need}, each of the terms in the string $y^{n+1}$ gets
applied to an element of positive depth; hence by~\eqref{d.xys},
we can repeatedly cancel sequences $x\,y$ at the interface between
the strings $x^j$ and $y^{n+1}.$
The outcome of these cancellations depends on the
relation between $j$ and $n$ in~\eqref{d.need}.
Suppose first that $j>n.$
Then the above repeated cancellations
turn the left-hand side of~\eqref{d.need} into
$y^j x^{j-n-1} x^n (s),$ which equals the right-hand, as desired.

If, rather, $j\leq n,$ then the cancellation mentioned
turns the left-hand side of~\eqref{d.need} into
$y^j y^{n+1-j} x^n (s)=y^{n+1} x^n (s),$ which by~\eqref{d.d>n}
is the value we have assigned to $y(s).$
We want to compare this with
the right-hand side of~\eqref{d.need}.
To do so, let us write that expression as $y (y^{j-1} x^{j-1})(s).$
Since $j<n,$ and $n$ is chosen as in~\eqref{d.d>n},
we have $d(x^{j-1}(s)) = j-1,$ hence by~\eqref{d.y_on_xS},
$d(y^{j-1} x^{j-1}(s)) = 0.$
So what we need to show is that $y,$ when applied to
the depth-zero element $y^{j-1} x^{j-1}(s),$
gives the same output as when applied to the depth-zero element $s.$
I claim that this will again hold by~\eqref{d.d>n}.
To see that, we have to know how the element $y^{j-1} x^{j-1}(s),$
and in particular, its depth, behave under $n$ successive
applications of $x.$
Now by~\eqref{d.xys}, the first $j{-}1$ of these $n$
applications of $x$ simply strip away
the same number of $\!y\!$'s, increasing $d$ by $1$ at each step.
Hence the result of applying $x^{j-1}$ to $y^{j-1} x^{j-1}(s)$
is $x^{j-1}(s);$ so from that point on, we get the same outputs
as when we apply the corresponding powers of $x$ to $s.$
So since~\eqref{d.d>n} applies to the evaluation of $y(s),$
it also applies to the evaluation of $y(y^{j-1} x^{j-1}(s)),$
and the resulting values are the same, as desired.

Finally, suppose we are in the case where $d(x^n(s)) = n$ for all $n.$
Then by~\eqref{d.if_d=n}, $y(s)$ will be a $\!y\!$-stable element,
i.e., will satisfy~\eqref{d.stable}.
This implies that it is fixed under $y^j x^j,$
so the result of applying the left-hand
side of~\eqref{d.yjxjy=} to $s$ is $y(s).$
When we apply the right-hand side of~\eqref{d.yjxjy=},
$y^j x^{j-1},$ to $s,$ if we write that operation
as $y(y^{j-1} x^{j-1}),$  then the factor $y^{j-1} x^{j-1}$ will
take $s$ to an element $s'$ again having $d(s') = 0.$
Moreover, by repeated application of~\eqref{d.xys}, $s'$ will
again satisfy $d(x^n(s')) = n$ for all $n,$ and will
have the same image as $s$ under $x^{j-1},$ hence will
be \mbox{$\!\approx\!$-equivalent} to $s.$
So by the last condition of~\eqref{d.if_d=n},
the images of $s$ and $s'$ under $y$ are the same.
So the two sides of~\eqref{d.yjxjy=}, applied to $s,$ give the
common value $y(s) = y(s'),$ completing the proof of the theorem.
\end{proof}

Pace Nielsen has pointed out that the proof of Lemma~\ref{L.bigcap}
can be modified so that the only relations on $x$ and $y$
called on are $xyx=x$ and $x^{n-1}y^n x^n = y x^n$ for
all $n>0$ (the $(1,1,1)$ and $(n-1,n,n)$
cases of~\eqref{d.many_rels}).
Namely, if we think of $x^{n-1}y^n x^n = y x^n$ as saying that
$x^{n-1}y^n$ and $y$ have the same effect on elements of $x^n S,$
then applied to elements $s\in\bigcap_{n\geq 0} x^n(S),$
this brings us directly to the relation $x^{n-1} y^n s = ys$ used in
that proof; after which only the relation $xyx=x$ is used.
(Contrast this with the second assertion
of Proposition~\ref{P.unit_reg},
which shows that we cannot get such a result using only the family
of relations $x^n y^n x^n = x^n$ and $y^n x^n y^n = y^n.)$
Consequently, Theorem~\ref{T.bigcap} can be strengthened to assert
the equivalence of three conditions on an endomap $x$ of a set:
(i)~condition~\eqref{d.bigcap},
(ii)~the existence of a strong inner inverse to $x,$ and
(iii)~the existence of an element $y$ satisfying the
relations $xyx=x,$ and $x^{n-1}y^n x^n = y x^n$ for all $n>0.$
The relations of~(iii) are strictly weaker than $y$ itself being a
strong inner inverse to $x,$ as may be seen from the
construction, in~\S\ref{S.var} below, of the
element $y'$ of~\eqref{d.y'}, which by~\eqref{d.y'_like_y} satisfies
those relations, but by~\eqref{d.neq} does not satisfy $y'=y'xy'.$

It is natural to ask whether~\eqref{d.bigcap} is also a sufficient
condition for an endomorphism $x$ of a vector space $V$ to
have a strong inner inverse in $\r{End}(V).$
Nielsen (personal communication) has an example showing that it is not.

\section{A closer look at $k\,\F$}\label{S.more}

In this and the next section we shall, for conceptual simplicity,
assume our division ring $D$ is a field $k,$
so that actions of $D\,\F$ on $\!D\!$-vector spaces $V,$
which for general $D$ are not actions by $\!D\!$-vector-space
endomorphisms, become actions of $k\,\F$ by
$\!k\!$-vector-space endomorphisms.
But the reader will see that virtually nothing about $k$
is used, so that the corresponding statements for $D\,\F,$
if desired, are available.

So let $k$ be a field, and
$V$ a countable-dimensional $\!k\!$-vector space with basis
$\{b_{n,h}\mid 1\leq h\leq n\},$ and let $x$
and $y$ be the endomorphisms of $V$ defined by
\begin{equation}\begin{minipage}[c]{35pc}\label{d.x,y}
\hspace*{-.15em}$x(b_{n,h})\ =\ b_{n,h+1}$\quad if
$h<n,\quad$ while\quad $x(b_{n,n})\ =\ 0,$\\
$y(b_{n,h})\ =\ b_{n,h-1}$\quad if\,
$h>1,$\quad\ while\,\quad $y(b_{n,1})\ =\ 0.$
\end{minipage}\end{equation}
By Theorem~\ref{T.main}, the monomials on the left-hand
side of~\eqref{d.many_rels} form a $\!k\!$-basis of the
subalgebra of $\r{End}(V)$ generated by $x$ and $y,$ and so, in
particular, are $\!k\!$-linearly independent.

But letting $V_n$ denote, for each $n,$ the subspace of $V$ spanned
by $b_{n,1},\dots,b_{n,n},$ consider the following four
monomials in $x$ and $y,$ and their actions on such a space $V_n:$
\begin{equation}\begin{minipage}[c]{35pc}\label{d.1,xy,yx,xyyx}
\hspace*{-.15em}$1$ fixes all $b_{n,h},$\\
$xy$ annihilates $b_{n,1},$ and fixes all other $b_{n,h},$\\
$yx$ annihilates $b_{n,n},$ and fixes all other $b_{n,h},$\\
$xyyx$ annihilates
$b_{n,1}$ and $b_{n,n},$ and fixes all other $b_{n,h}.$
\end{minipage}\end{equation}
These descriptions suggest that $1 + xyyx = xy + yx$ on each $V_n.$

What is wrong here?

The relation $1 + xyyx = xy + yx$
does in fact hold on all $V_n$ with $n>1.$
But for $n=1,$ the above considerations implicitly ``double-count''
the basis element $b_{1,1}$ in looking at the effect of $xyyx.$
On $V_1,$ the elements $xy,$ $yx$ and $xyyx$ all act
as $0,$ while $1$ does not;
so the asserted relation fails on $V_1$ -- and only there.

This suggests that operators on $V$ that,
like $1 + xyyx - xy - yx,$ have ``small'' images,
may be of interest in understanding $k\,\F.$
Let us define on $V,$ for each $i\geq 1,$ operators
$\ell_i$ and $r_i,$ by
\begin{equation}\begin{minipage}[c]{35pc}\label{d.l_i,r_i}
\hspace*{-.15em}$\ell_i$ fixes, for each $n\geq i,$ the
element $b_{n,i},$ and annihilates all other\\
\hspace*{1em}elements $b_{n,h}$ $(n\geq 1,\ 1\leq h\leq n),$\\
$r_i$ fixes, for each $n\geq i,$ the
element $b_{n,n+1-i},$ and annihilates all other\\
\hspace*{1em}elements $b_{n,h}$ $(n\geq 1,\ 1\leq h\leq n).$
\end{minipage}\end{equation}
Here $\ell$ and $r$ are mnemonic for ``left'' and ``right'', since
if we list the basis of $V_n$ as $b_{n,1},\dots,b_{n,n},$
then $\ell_i$ projects to the $\!i\!$-th basis element
from the left (if any), and $r_i$ to the $\!i\!$-th from the right
(if any).

These operators are represented by elements of $k\,\F.$
Namely, identifying elements of that algebra with their actions on $V,$
it is not hard to check that for all $i\geq 1,$
\begin{equation}\begin{minipage}[c]{35pc}\label{d.l_i,r_i=}
\hspace*{-.15em}$\ell_i\ =\ x^{i-1} y^{i-1} - x^i y^i,$\\
$r_i\ =\ y^{i-1} x^{i-1} - y^i x^i.$
\end{minipage}\end{equation}
Elements with still smaller images are
given by products of the above operators.
Namely, for each $i,j\geq 1,$
\begin{equation}\begin{minipage}[c]{35pc}\label{d.r_il_j}
$\ell_i r_j$ fixes the single basis
element $b_{i+j-1,i},$ and annihilates all the other $b_{n,h}$
$(1\leq h\leq n).$
\end{minipage}\end{equation}

Let us note some relations which $x,$ $y,$ and the $\ell_i$
and $r_i$ satisfy.
First, the $i=1$ cases of~\eqref{d.l_i,r_i=} give the following
formulas, which can be applied to reduce any string of
$\!x\!$'s and $\!y\!$'s which
contains both letters to a linear combination of shorter strings
in $x,$ $y,$ $\ell_1$ and $r_1:$
\begin{equation}\begin{minipage}[c]{35pc}\label{d.xy,yx}
$xy\ =\ 1 - \ell_1,$\quad $yx\ =\ 1 - r_1.$
\end{minipage}\end{equation}

In the next list of relations
we make the temporary convention that $\ell_0$
and $r_0$ represent the zero operator.
Then we have the following equalities for all $i\geq 1$
(in which the $i=1$ cases each say that a certain product is $0).$
\begin{equation}\begin{minipage}[c]{35pc}\label{d.l_ix_etc}
$\ell_i\,x = x\,\ell_{i-1},$ \quad
$r_{i-1}\,x = x\,r_i,$ \quad
$\ell_{i-1}\,y = y\,\ell_i,$ \quad
$r_i\,y = y\,r_{i-1}.$
\end{minipage}\end{equation}
These relations are easily deduced from~\eqref{d.x,y}
and~\eqref{d.l_i,r_i}, and allow us to reduce expressions in
$x,$ $y$ and the $\ell_i$ and $r_i$ to linear combinations of monomials
in which no $\ell_i$ or $r_i$ occurs to the left of an $x$ or $y,$ and
which also contain no products $x\,r_1$ or $y\,\ell_1.$

The $i=1$ cases of the second and third relations
of~\eqref{d.l_ix_etc} can be seen to generalize to
\begin{equation}\begin{minipage}[c]{35pc}\label{d.=0}
$x^i r_i\ =\ 0,\quad y^i\ell_i\ =\ 0\quad (i\geq 1).$
\end{minipage}\end{equation}

Finally, from~\eqref{d.l_i,r_i} it is easy to see that
\begin{equation}\begin{minipage}[c]{35pc}\label{d.ll,rr,rl}
\hspace*{-.15em}$\ell_i^2 = \ell_i,$ while
$\ell_i\ell_j = 0$ for $i\neq j,$\\[.2em]
$r_i^2 = r_i,$ while
$r_ir_j = 0$ for $i\neq j,$\\[.2em]
$\ell_i r_j = r_j \ell_i$ for all $i,j.$
\end{minipage}\end{equation}

Using~\eqref{d.xy,yx}-\eqref{d.ll,rr,rl}, we can reduce any
element of $k\,\F$ to a $\!k\!$-linear combination
of elements of the following four sorts.
\begin{equation}\begin{minipage}[c]{35pc}\label{d.y1x}
$y^m$ $(m\geq 1),$\quad $1,$\quad $x^m$ $(m\geq 1).$
\end{minipage}\end{equation}
\begin{equation}\begin{minipage}[c]{35pc}\label{d.y1x_l}
$y^m\,\ell_i$ $(i>m\geq 1),$\quad $\ell_i$ $(i\geq 1),$ \quad
$x^m\,\ell_i$ $(i,m\geq 1).$
\end{minipage}\end{equation}
\begin{equation}\begin{minipage}[c]{35pc}\label{d.y1x_r}
$y^m\,r_i$ $(i,m\geq 1),$\quad $r_i$ $(i\geq 1),$ \quad
$x^m\,r_i$ $(i>m\geq 1).$
\end{minipage}\end{equation}
\begin{equation}\begin{minipage}[c]{35pc}\label{d.y1x_rl}
$y^m\,r_i\,\ell_j$ $(i,m\geq 1;\,j>m),$\quad
$r_i\,\ell_j$ $(i,j\geq 1),$\quad
$x^m\,r_i\,\ell_j$ $(j,m\geq 1;\,i>m).$
\end{minipage}\end{equation}

In fact, we have

\begin{lemma}\label{L.basis}
The elements listed in~\eqref{d.y1x}-\eqref{d.y1x_rl}
form a $\!k\!$-basis of $k\,\F.$
\end{lemma}

\begin{proof}
We have seen that the elements~\eqref{d.y1x}-\eqref{d.y1x_rl}
span $k\,\F,$ so it will
suffice to show that every nontrivial $\!k\!$-linear combination
$f$ of those elements has nonzero action on $V.$

Suppose first that $f$ involves at least one of the
elements in~\eqref{d.y1x}, say $u\in\{y^m,\,1,\,x^m\},$
with nonzero coefficient in $k.$
I claim that we can find an element $b_{n,i}$ in our basis for $V$
which is annihilated by all the monomials occurring in $f$ that lie
in~\eqref{d.y1x_l}-\eqref{d.y1x_rl}, but not by $u.$
Indeed, for each $n>0,$ each term in~\eqref{d.y1x_l}-\eqref{d.y1x_rl}
has nonzero action on at most one of $b_{n,1},\dots,b_{n,n},$
while the number of members of each such set annihilated by
our element $u$ is bounded, independent of $n,$
by the exponent on $x$ or $y$ in $u$ (if any).
Hence for large enough $n,$ there exists a $b_{n,i}$
which is neither an element on which the finitely many terms of $f$
in~\eqref{d.y1x_l}-\eqref{d.y1x_rl}
have nonzero value, nor one annihilated by $u;$
let us choose any such $b_{n,i}.$
Note that each member of~\eqref{d.y1x} other than $u$
which is nonzero on $b_{n,i}$ carries it to a basis
element other than $u(b_{n,i}),$ since
by~\eqref{d.x,y}, distinct members of~\eqref{d.y1x}
shift second subscripts of the $b_{n,i}$ by different amounts.
Hence in $f(b_{n,i}),$ the basis-element $u(b_{n,i})$ has nonzero
coefficient; so $f$ has nonzero action on $V.$

Suppose, next, that $f$ involves no members of~\eqref{d.y1x},
but has some member of~\eqref{d.y1x_l} or~\eqref{d.y1x_r}
with nonzero coefficient.
Assume without loss of generality that it involves
an element $u$ of~\eqref{d.y1x_l}, and let $\ell_i$ be the
$\!\ell\!$-factor occurring.
Thus, $u$ acts nontrivially only on basis elements
$b_{n,i}$ $(n\geq i),$
and I claim we can find some $n$ for which $u,$ but none
of the elements of~\eqref{d.y1x_r} or~\eqref{d.y1x_rl} occurring
in $f,$ does so.
Note that $u$ annihilates $b_{n,i}$ for at most finitely many $n.$
Each element of~\eqref{d.y1x_r} acts nontrivially on members of
our basis that are a fixed distance from the {\em right} end of
the families $\{b_{n,1},\dots,b_{n,n}\},$ so for $n$ large
enough, none of the finitely many such elements
which occur in $f$ will act nontrivially on the $\!i\!$-th element
from the {\em left}.
Moreover, each member of~\eqref{d.y1x_rl}
is nonzero on $b_{n,i}$ for at most one $n.$
Hence for all but finitely many $n,$ the element $b_{n,i}$
has nonzero image under $u$ but not under any of the
members of~\eqref{d.y1x_r} or~\eqref{d.y1x_rl} occurring in $f;$
let us pick such a $b_{n,i}.$
Of the other members of~\eqref{d.y1x_l} occurring in $f,$
those that involve $\ell_j$ for some $j\neq i$ annihilate $b_{n,i},$
while those that involve $\ell_i$ but begin with a different
factor $x^m,$ $1$ or $y^m$ send $b_{n,i}$ to
a different basis element.
So again, the basis element $u(b_{n,i})$ has nonzero coefficient
in $f(b_{n,i}),$ so $f$ has nonzero action on~$V.$

Finally, if $f$ involves no members of~\eqref{d.y1x}-\eqref{d.y1x_r},
we take a member $u$ of~\eqref{d.y1x_rl} that it involves, and let
$r_i \ell_j$ be the factors other than a power of $x$ or $y$
in that operator.
Then $u(b_{i+j-1,j})\neq 0,$ and we easily check that
all other elements of~\eqref{d.y1x_rl} (which are the only
other terms that can appear in $f)$ either
annihilate $b_{i+j-1,j},$ or send it to a basis element
different from the one to which $u$ sends it.
Hence $f(b_{i+j-1,j})\neq 0,$
so in this case, too, $f$ has nonzero action on~$V,$
completing the proof of the lemma.
\end{proof}

Although $\F$ has trivial center, $k\,\F$ has nontrivial
central idempotents.
Namely, for each $n\geq 1,$ the element
\begin{equation}\begin{minipage}[c]{35pc}\label{d.p_n}
$p_n\ =\ r_n\ell_1 + r_{n-1}\ell_2 + \dots + r_1\ell_n$
\end{minipage}\end{equation}
acts on $V$ by projection to $V_n;$ and since all members of
$k\,\F$ carry each $V_m$ into itself,
$p_n$ commutes with all such elements.
(For results on the centers of monoid algebras of
free inverse monoids on more than one generator, see~\cite{C+M}.)

The above material on the structure of $k\,\F$ has considerable
overlap with \cite[\S 4]{PA+KG}.
In particular, our $\ell_i r_j$ and $p_n$ would be, in the
notation of display~(4.4) and Lemma~4.3 of that
note, $q_{-i+1,j-1},$ and $h_{n-1};$
and Proposition~4.5 of that paper
shows, inter alia, that the latter elements generate the
socle of $k\,\F$ (which is equivalent to saying that
the socle is spanned over $k$ by the former elements).
\vspace{0.5em}

In the above discussion, we have been regarding $k\,\F$ as an
algebra of operators on the space $V=\bigoplus V_n,$
but the results on its structure that we have
deduced necessarily hold for the abstract
algebra $k\,\F;$ so we can look at the elements considered
above in connection with any action of $k\,\F$ on any vector space.
For instance, if $x$ is an {\em invertible} endomorphism of a
vector space $V,$ and $y$ its inverse, then we get an action
in which, it is easy to see, all $\ell_i$ and $r_i$ act trivially
(i.e., as the zero operator).
If $x$ acts by a one-to-one but not
invertible endomorphism of a space $V,$ and we let
$y$ act by a left inverse thereof, we get an action
under which the $r_i$ are trivial, but not all the $\ell_i;$
and we have the obvious dual
statement if $x$ is surjective but not one-to-one.
In these cases, all the products $r_i\,\ell_j,$
and hence the central idempotents $p_n,$ act trivially.

Note that there is a homomorphism $\F\to\Z$
taking $x$ to $1$ and $y$ to $-1,$ and that this induces
a grading on $k\,\F,$ under which the
basis elements~\eqref{d.y1x}-\eqref{d.y1x_rl} are homogeneous:
those whose expressions begin with $y^m$ have degree $-m,$
those beginning with $x^m$ have degree $+m,$
and those beginning with neither have degree $0.$
The homogeneous component of $k\,\F$ of degree $0$ is commutative,
generated over $k$ by the idempotent elements $\ell_i$ and $r_i.$

The following result, which relates the behavior of
$r_i$ and $\ell_i$ with that of $x$ under an arbitrary
action of $k\,\F$ on a $\!k\!$-vector space,
will be used in the next section.

\begin{lemma}\label{L.x_alone}
Let $V$ be a $\!k\!$-vector space given with a left action
of $k\,\F,$ and $i$ a positive integer.
Then for each $i\geq 1,$
\vspace{.2em}

\textup{(i)} \ The action of
the projection map $r_i$ on $V$ annihilates the subspace
$\r{ker}(x^{i-1}),$ and has as image a complement of that
subspace in the \textup{(}generally larger\textup{)}
subspace $\r{ker}(x^i).$
\vspace{.2em}

\textup{(ii)} The action of
the projection map $\ell_i$ on $V$ annihilates the
subspace $\r{im}(x^i),$ and has as image a complement of that
subspace in the \textup{(}generally larger\textup{)}
subspace $\r{im}(x^{i-1}).$
\vspace{.2em}

\textup{(}In these statements, $x^0$ is understood to be
the identity operator.
Thus, in the $i=1$ case of~\textup{(i)},
since $\r{ker}(\r{id})=0,$ the conclusion simply
means that the projection $r_1$ has image $\r{ker}(x);$
and in~\textup{(ii)}, since $\r{im}(\r{id})=V,$ the
conclusion means that $\ell_1$ is a projection of $V$
along $\r{im}(x)$ onto a complement of that subspace.\textup{)}

The corresponding statements hold with $x$
replaced by $y,$ and the roles of $\ell_i$ and $r_i$ interchanged.
\end{lemma}

\begin{proof}
Since $r_i$ and $\ell_i$ are idempotent elements of $k\,\F,$
they act on $V$ by projection operators.

Since $r_i = y^{i-1} x^{i-1} - y^i x^i,$ this element is right
divisible by $x^{i-1},$ hence annihilates $\r{ker}(x^{i-1});$ while
its image lies in $\r{ker}(x^i)$ by the first equation of~\eqref{d.=0}.
Since it is a projection, to show that its image
is a complement of $\r{ker}(x^{i-1})$ in $\r{ker}(x^i),$ it remains
only to show that the only elements of $\r{ker}(x^i)$ annihilated
by $r_i$ are the elements of $\r{ker}(x^{i-1}).$
And indeed, if $v\in\r{ker}(x^i)$ is annihilated by
$r_i = y^{i-1} x^{i-1} - y^i x^i,$ then it is
annihilated by $y^{i-1} x^{i-1},$ hence by that element's left
multiple $x^{i-1} y^{i-1} x^{i-1} = x^{i-1},$ as required.

Turning to $\ell_i,$ this annihilates
$\r{im}(x^i)$ by repeated application of the
first equation of~\eqref{d.l_ix_etc}.
Since $\ell_i = x^{i-1} y^{i-1} - x^i y^i,$ it is
left divisible by $x^{i-1},$ so its image lies in $\r{im}(x^{i-1}).$
Thus, as above, it remains to show that any element of $\r{im}(x^{i-1})$
annihilated by $\ell_i$ lies in $\r{im}(x^i).$
And indeed, the condition that an element
$x^{i-1}(v)\in\r{im}(x^{i-1})$ be annihilated
by $\ell_i$ is $x^{i-1} y^{i-1} x^{i-1}(v) - x^i y^i x^{i-1}(v) = 0.$
Applying~\eqref{d.many_rels} to each term of this relation,
it becomes $x^{i-1}(v) - x^i y(v) = 0,$ showing that
$x^{i-1}(v)$ indeed lies in $\r{im}(x^i).$

The final assertion is clear by symmetry.
(We will not need it below.)
\end{proof}

The referee has asked whether the ideas of this note might be
of use in connection with the question of whether all von~Neumann
regular satisfy the {\em separativity} condition on their monoids of
finitely generated projective modules (see~\cite{PA}).
My one attempt to contribute to that problem was~\cite{u_inner_inv},
which examines the ring $R'$ obtained by universally adjoining to
a $\!k\!$-algebra $R$ an inner inverse to an element $x\in R.$
The hope was that by applying that construction recursively to
an algebra with non-separative monoid of projectives, one might
get a regular $\!k\!$-algebra with the same property.
But though results were obtained on the element-structures of
$\!R'\!$-modules $M\otimes_R R',$ there was no evident way to go from
these to results on the monoid of finitely generated projectives.

One might similarly study the result of adjoining to
a $\!k\!$-algebra $R$ a universal
{\em strong} inner inverse to an element $x\in R,$ with the same goal.
I do not know whether such an approach
would have a better chance of success.

\section{Inner inverses satisfying other systems of relations}\label{S.var}

In this section, we will regard the monoid
relations~\eqref{d.many_rels},
\eqref{d.few_rels1}, \eqref{d.few_rels2} by which we have
defined $\F$ as $\!k\!$-algebra relations defining $k\,\F.$
It is natural to ask whether the relations satisfied by $x$
and $y$ in $k\,\F$ comprise the ``strongest'' family of
$\!k\!$-algebra relations that one can force an inner inverse $y$ of
a general element $x$ of a ``good'' regular $\!k\!$-algebra
(such as an infinite product of full matrix algebras over $k)$ to have.

An easy observation which suggests a negative answer is the following:
It is known that for all $n,$ the ring $R$ of $n\times n$
matrices over $k$ is {\em unit regular,} i.e., that
for every $x$ there exists an {\em invertible}
$y$ such that $xyx = x$ \cite[Lemma~1.3]{KG}.
Now if $x$ is not invertible and $y$ is, we can't have $yxy = y;$
so for noninvertible $x,$ the condition that $y$ be invertible is
incompatible with the relations~\eqref{d.many_rels}.
Unfortunately, the condition that $y$ be invertible is not
a $\!k\!$-algebra relation in $x$ and $y;$ so the above does not
contradict the possibility that the relations defining $k\,\F$ might be
the strongest set of $\!k\!$-algebra relations in those two
elements that we can force an inner inverse $y$ to have in $R.$

We shall see at the end of this section that a modification of
the above idea does work.
But let us first show that the set of $\!k\!$-algebra relations
defining $k\,\F,$ even if not a {\em greatest} element in the
partially ordered set of families of relations that can be
so forced, is a {\em maximal} element.
For this, we will need

\begin{corollary}[to Lemma~\ref{L.x_alone}]\label{C.x_alone}
An action of $k\,\F$ on a vector space $V$ is faithful
if and only if for all positive integers $i,$
\begin{equation}\begin{minipage}[c]{35pc}\label{d.x_alone}
$\r{im}(x^i)\,\cap\,\r{ker}(x)\ \neq
\ \r{im}(x^{i-1})\,\cap\,\r{ker}(x).$
\end{minipage}\end{equation}
\end{corollary}

\begin{proof}
Since the left-hand side of~\eqref{d.x_alone} is contained
in the right-hand side, it suffices to show that the
given action is faithful if and only if for every $i,$
there is an element of the right-hand side of~\eqref{d.x_alone}
that is not in the left-hand side.

It is easy to verify that every nonzero two-sided ideal of $k\,\F$
contains at least one of the elements~\eqref{d.y1x_rl}; and
with the help of~\eqref{d.l_ix_etc}, one can deduce that
every such ideal contains an element $\ell_i r_1$ $(i\geq 1).$
We shall now show that for each $i,$
the existence of an element belonging to the right-hand
side of~\eqref{d.x_alone} but not the left is equivalent to the
condition that $\ell_i r_1$ act nontrivially on $V,$
i.e., {\em not} be in the kernel of our action.
Thus, the inequalities~\eqref{d.x_alone} will hold for all $i$
if and only if that kernel is zero, giving the desired result.

Suppose first that some $v\in V$ belongs to
$\r{im}(x^{i-1})\,\cap\,\r{ker}(x)$ but not to
$\r{im}(x^i)\,\cap\,\r{ker}(x).$
Since $v\in\r{ker}(x),$ we see from the parenthetical
statement in Lemma~\ref{L.x_alone}
that it is fixed under $r_1,$ while since it lies
in $\r{im}(x^{i-1})$ but not in $\r{im}(x^i),$ statement~(ii) of
that lemma shows that it is not annihilated by $\ell_i.$
Hence $\ell_i r_1(v)\neq 0,$ so $\ell_i r_1$ indeed acts nontrivially.

Conversely, suppose $\ell_i r_1$ acts nontrivially on some $v\in V.$
Then I claim that $\ell_i r_1(v)$ lies in the right-hand side
of~\eqref{d.x_alone} but not in the left-hand side.
Indeed, since $\ell_i r_1(v)$ is a nonzero
element of $\ell_i(V),$ Lemma~\ref{L.x_alone}(ii) tells us that it
belongs to $\r{im}(x^{i-1})$ but not to $\r{im}(x^i),$
while since $\ell_i r_1 = r_1\ell_i,$ it also lies
in $r_1(V)=\r{ker}(x),$ as required.
\end{proof}

We deduce

\begin{proposition}\label{P.max}
Let $R$ be a $\!k\!$-algebra, and $\varphi$ and $\varphi'$
$\!k\!$-algebra homomorphisms $k\,\F\to R$
such that $\varphi(x)=\varphi'(x).$
Then $\varphi$ is one-to-one if and only if $\varphi'$ is.

Hence no {\em proper} homomorphic image $(k\,\F)/I$
has the property that for every element $x$ in a direct product
$R$ of full matrix rings over $k,$ there exists $y\in R$ which,
with $x,$ satisfies the relations of $(k\,\F)/I.$

Thus, the set of $\!k\!$-algebra relations satisfied by $x$ and $y$
in $k\,\F$ is maximal among sets $S$ of $\!k\!$-algebra
relations in two noncommuting indeterminates such that
for every element $x$ in a direct product $R$ of full
matrix algebras, one can find a $y\in R$
such that $x$ and $y$ together satisfy $S.$
\end{proposition}

\begin{proof}
Embedding $R$ in an algebra of the form $\r{End}(V),$
we get from the preceding corollary a condition
which is necessary and sufficient
both for $\varphi$ and for $\varphi'$ to be one-to-one,
in terms of the action of $x$ on $V.$
Since $\varphi(x)=\varphi'(x),$ this yields the first assertion.

We know from the second paragraph of Theorem~\ref{T.main}
that there exist elements $x$ and $y$ in an infinite direct product of
matrix rings over $k$ which satisfy the relations of $k\,\F$ but no
others, hence the preceding assertion
shows that for such $x,$ there exists no
$y'$ such that $x$ and $y'$ satisfy the relations of a proper
homomorphic image $(k\,\F)/I$ of $k\,\F.$

The third assertion then follows.
\end{proof}

One can easily strengthen the above proof to show that
for any $\varphi$ and $\varphi'$ as in the first sentence of the
proposition, their kernels have the same intersection with the ideal
of $k\,\F$ spanned by the elements~\eqref{d.y1x_rl}
(the socle of $k\,\F).$
It seems likely that these kernels must in
fact be equal; I leave this for others to investigate.
\vspace{.4em}

Let us now prove, on the other hand, that the set of relations
satisfied by $x$ and $y$ in $k\,\F$ is not the only maximal set
of $\!k\!$-algebra relations that can be forced in this way.
We will use the idea of the second paragraph
of this section in our example, but instead of
trying to make $y$ everywhere invertible, we shall only do this
on $\!x\!$-basic subspaces of one chosen dimension $m.$
To see how, let us again think in terms the action of $k\,\F$
on the vector space $V=\bigoplus V_n$
spanned by elements $b_{n,h},$ as in~\S\ref{S.more}.
If we let
\begin{equation}\begin{minipage}[c]{35pc}\label{d.y'}
$y'\ =\ y\,+\,x^{m-1} \ell_1 r_m,$
\end{minipage}\end{equation}
we see that $y'$ acts like $y$ on the spaces $V_n$ with $n\neq m$
(since $\ell_1 r_m$ acts trivially on those spaces),
but that on $V_m,$ it cyclically permutes the $m$ basis
elements $b_{m,i}.$
Since $x$ does not act invertibly on $V_m,$ neither
does $y'xy',$ so since $y'$ does, we have
\begin{equation}\begin{minipage}[c]{35pc}\label{d.neq}
$y'\ \neq\ y'x\,y'.$
\end{minipage}\end{equation}

On the other hand, the reader can easily verify,
by calculation in $\r{End}(V),$ that
for all natural numbers~$i,$
\begin{equation}\begin{minipage}[c]{35pc}\label{d.y'_like_y}
$(y')^i\,x^i\,=\,y^i\,x^i,$ \quad and \quad $x^i\,(y')^i\,=\,x^i\,y^i.$
\end{minipage}\end{equation}
Hence all $\ell_i$ and $r_i$
are expressible in terms of $x$ and $y'$ by the same
formulas~\eqref{d.l_i,r_i=} that express them in terms of $x$ and $y.$
In particular, the subalgebra of $k\,\F$ generated by $x$ and $y'$
contains all $\ell_i$ and $r_i,$
and so contains $y=y'-x^{m-1} \ell_1 r_m$
(see~\eqref{d.y'}), so it is all of $k\,\F.$
From the $i=1$ case of~\eqref{d.y'_like_y}, we also
see that $x=xy'x;$ i.e., $y'$ is an inner inverse to $x.$

The difference in behavior between $y$ and $y'$ can
be attested concretely by the easily checked relations
\begin{equation}\begin{minipage}[c]{35pc}\label{d.y'_vs_y}
$(y')^m\,p_m\ =\ p_m,$\qquad $y^m\,p_m = 0.$
\end{minipage}\end{equation}
Since $p_m$ is an expression in the $\ell_i$ and $r_i,$
which can be expressed in terms of $x$ and $y'$ by the same formulas
that express them in terms of $x$ and $y,$ the
equations of~\eqref{d.y'_vs_y} can be regarded as contrasting relations
satisfied in $k\,\F$ by $x$ and $y'$
and by $x$ and $y,$ respectively.

Thus, we get

\begin{proposition}\label{P.y'_like_y}
In the $\!k\!$-algebra $k\,\F,$ the element $y'$ given
by~\eqref{d.y'} is another inner inverse to $x,$ which
together with $x$ generates the whole algebra, but which,
by~\eqref{d.y'_vs_y},
satisfies with $x$ a family of relations incomparable
with the family of relations satisfied by $y.$
Hence these two sets of relations are distinct maximal
elements among the families $S$ of $\!k\!$-algebra relations
referred to in the last sentence of Proposition~\ref{P.max}.\qed
\end{proposition}

This leaves open the question of
whether every algebra presented by generators $x$ and $y,$
and a maximal set of relations which include the
relation $x=xyx,$ and which can be ``forced'' in the
sense we have been discussing,
must be isomorphic to $k\,\F$ by an isomorphism fixing $x.$
In addition to questions about relations satisfied by $x$
and $y$ alone, one might want to look at variants, such as
what relations can be forced on an {\em invertible}
inner inverse $y$ to an element $x,$ when one exists.

\section{Acknowledgements}\label{S.akn}
I am indebted to Pace Nielsen, Kevin O'Meara, and Janez \v{S}ter
for stimulating interchanges leading to this material,
to Ben Steinberg and Mark Lawson for information on inverse monoids,
to Warren Dicks, Ken Goodearl, T.\,Y.\,Lam
and Raphael Robinson for further help, and to the
referee for some interesting questions, including two discussed
following Proposition~\ref{P.unit_reg}, and the one mentioned
at the end of~\S\ref{S.more}.

\end{document}